\documentclass[11pt,a4paper,reqno,oneside]{amsart}
\usepackage[a4paper]{geometry}
\makeatletter
\newcommand{\authorfootnotes}{\renewcommand\thefootnote{\@fnsymbol\c@footnote}}%
\makeatother
\usepackage[utf8]{inputenc}
\usepackage{amsthm,mathtools,amsfonts,amsmath,amssymb,xfrac,enumitem,graphicx,mathrsfs,authblk,comment,braket}
\usepackage[foot]{amsaddr}
\usepackage{url}

\usepackage{etoolbox}% http://ctan.org/pkg/etoolbox for AtEndEnvironment

\usepackage[unicode=true,
bookmarks=true,bookmarksopen=false,
breaklinks=false,pdfborder={0 0 0},colorlinks=true]
{hyperref}
\usepackage{xcolor}
\definecolor{cblue}{rgb}{0.16, 0.32, 0.75}
\definecolor{cred}{rgb}{0.7, 0.11, 0.11}
\hypersetup{%
  ,linkcolor=cred
  ,citecolor=cblue
  ,urlcolor=black
}

\usepackage[
backend=bibtex,
style=numeric-comp,
date=year,
giveninits=true,
maxbibnames=299,
natbib=true,
doi=false,
isbn=false,
url=false,
sorting=nyt
]
{biblatex}
\bibliography{singular_perturbations.bib}
\renewbibmacro{in:}{}
\DeclareFieldFormat[misc]{title}{``#1''\isdot}
\DeclareFieldFormat{journaltitle}{#1\isdot}
\DeclareFieldFormat[article,periodical]{volume}{\mkbibbold{#1}}
\DeclareFieldFormat{pages}{#1}
\AtEveryBibitem{%
  \clearfield{number}
  \clearlist{language}
}

\AtBeginDocument{\DeclareMathSymbol{:}{\mathpunct}{operators}{"3A}}
\def\colon{\mathrel{:}}

\newcommand{\lspan}[1]{\mathrm{span}\left\{#1\right\}}
\DeclareMathOperator{\dom}{dom}

\renewcommand{\d}{\mathrm{d}}

\newcommand{\hilb}{\mathcal{H}}
\newcommand{\norm}[1]{\left\| #1 \right\|}

\newtheorem{theorem}{Theorem}[section]

\newtheorem{lemma}[theorem]{Lemma}
\newtheorem{proposition}[theorem]{Proposition}
\theoremstyle{definition}
\newtheorem{definition}[theorem]{Definition}
\newtheorem{assumption}[theorem]{Assumption}
\theoremstyle{remark}
\newtheorem{remark}[theorem]{Remark}

\numberwithin{equation}{section}

\newcommand{\triang}{\hfill$\triangle$}
\AtEndEnvironment{example}{\triang}
\AtEndEnvironment{remark}{\triang}
\AtEndEnvironment{construction}{\triang}

\usepackage{color}
\usepackage[normalem]{ulem}

\usepackage{marginnote}
\newcommand\blfootnote[1]{%
  \begingroup
  \renewcommand\thefootnote{}\footnote{#1}%
  \addtocounter{footnote}{-1}%
  \endgroup
}

\title[Global approximate controllability by form perturbations]{Global approximate controllability of quantum systems by form perturbations and applications}
\author{Aitor Balmaseda\textsuperscript{1,2}}
\author{Davide Lonigro\textsuperscript{3,4,5}}
\author{Juan Manuel Pérez-Pardo\textsuperscript{1}}

\address{\footnotesize\textsuperscript{1}Departamento de Matemáticas, Universidad Carlos III de Madrid, 28911 Madrid, Spain}
\address{\footnotesize\textsuperscript{2}Departamento de Análisis Matemático y Matemática Aplicada, Facultad de Ciencias Matemáticas, Universidad
Complutense de Madrid, 28040 Madrid, Spain}
\address{\footnotesize\textsuperscript{3}Department Physik, Friedrich-Alexander-Universität Erlangen-Nürnberg, Staudtstraße 7, 91058 Erlangen, Germany}
\address{\footnotesize\textsuperscript{4}Dipartimento di Matematica, Universit\`a degli Studi di Bari Aldo Moro, 70125 Bari, Italy}
\address{\footnotesize\textsuperscript{5}Istituto Nazionale di Fisica Nucleare, Sezione di Bari, 70126 Bari, Italy}

\email{\footnotesize\href{mailto:abalmase@math.uc3m.es}{\texttt{abalmase@math.uc3m.es}}}
\email{\footnotesize\href{mailto:davide.lonigro@fau.de}{\texttt{davide.lonigro@fau.de}}}
\email{\footnotesize\href{mailto:jmppardo@math.uc3m.es}{\texttt{jmppardo@math.uc3m.es}}}

\begin{document}

\maketitle
\thispagestyle{empty}

\vspace{-1.25cm}


\begin{abstract}
We provide sufficient conditions for the approximate controllability of infinite-dimensional quantum control systems corresponding to form perturbations of the drift Hamiltonian modulated by a control function. We rely on previous results on controllability of quantum bilinear control systems and obtain a priori $L^1$-bounds of the controls for generic initial and target states. We apply a stability result for the non-autonomous Schrödinger equation to extend the results to systems defined by form perturbations, including singular perturbations. 
As an application of our results, we prove approximate controllability of a quantum particle in a one-dimensional box with a point-interaction with tuneable strength at the centre of the box.
\end{abstract}

\section{Introduction}\label{sec:intro}

Being able to manipulate and steer the evolution of quantum states and systems is a key aspect in the development of quantum technologies. This is an interdisciplinary field that ranges from experimental physics and egineering to theoretical physics and mathematical analysis. Quantum Control Theory is the mathematical discipline that studies and characterises the properties of the evolution of quantum states under the action of external interactions and the feasibility of effectively controlling or guiding the evolution of the system. 
We refer to \cite{DAlessandro2007} for an introduction and to \cite{koch2022quantum} for a recent review on Quantum Control.

In this article we introduce a set of sufficient conditions for the global approximate controllability of infinite-dimensional quantum control systems defined by \textit{form perturbations} of a given self-adjoint operator $H_0$ bounded from below. That is, families of self-adjoint operators bounded from below on a Hilbert space $\hilb$ uniquely associated with a sesquilinear form defined as:
\begin{equation*}
  h_0+u(t)h_1,
\end{equation*}
where $h_0$ is the sesquilinar form uniquely associated with $H_0$, $h_1$ is a sesquilinear form, and $t\mapsto u(t)$ is a real-valued control function. We denote these as \textit{form bilinear} quantum control systems. Importantly, our setting is sufficiently flexible to include ``singular'' situations in which either $h_1$ is a \textit{singular} perturbation, not associated with any linear operator $H_1$ on $\hilb$ (for example, a point interaction in dimension one), or in which such an operator does exist but the sum $H_0+u(t)H_1$ is not well-defined, or at the very least not essentially self-adjoint, on a sufficiently large domain. The controllability results proven are \textit{global}, in the sense that no constraints on the choice of the initial and target states $\Psi_0,\Psi_1\in\hilb$ are needed, besides the natural constraint imposed by the unitarity of the evolution, $\norm{\Psi_0} = \norm{\Psi_1}$. The results involve control functions with predetermined maximum value $r>0$ in an arbitrarily large but finite control time $T>0$.

Global exact controllability, in contrast with what happens in the finite-dimensional case \cite{DAlessandro2007}, is not an appropriate notion of controllability in the context of quantum systems modelled in infinite-dimensional separable Hilbert spaces \cite{BallMarsdenSlemrod1982, Turinici2000} and studying the reachable set becomes a major difficulty. In this context approximate controllability and related notions, e.g.\ \cite{IbortPerezPardo2009}, provide a convenient setting. Exact controllability in dense subspaces and local exact controllability has been proven in particular cases \cite{Beauchard2005a, BeauchardCoron2006a, beauchard2010local, nersesyan2010global}. To obtain generic results on controllability on infinite-dimensional systems, similar to those characterised by the geometric control theory on finite dimensional systems is an area of active research. Of particular importance are the so-called Lie-Galerkin methods, by means of which sufficient conditions to achieve approximate controllability have been found \cite{ChambrionMasonSigalottiEtAl2009, MasonSigalotti2010, BoscainCaponigroChambrionEtAl2012, BoussaidCaponigroChambrion2013}. These consist in extending finite-dimensional results to the infinite dimensional setting. For doing this, one needs to guarantee that the inifinite-dimensional dynamics can be well-approximated by finite-dimensional ones. These methods have been applied, for instance, to the control of molecules \cite{BoussaidCaponigroChambrion2019, BoscainPozzoliSigalotti2021}.

To achieve the results about controllability on form bilinear control systems we base our analysis on existing results in the literature \cite{ChambrionMasonSigalottiEtAl2009,BoscainCaponigroChambrionEtAl2012,BoussaidCaponigroChambrion2013} about the approximate controllability of bilinear control systems. These are families of time-dependent Hamiltonians of the form $H_0+u(t)H_1$,
where $H_0$ is a self-adjoint operator possessing a complete set of eigenvectors, $H_1$ is a symmetric operator, and such that $H_0+u(t)H_1$ is essentially self-adjoint on a sufficiently large domain.
The sufficient conditions for the approximate controllability split in three different categories. First, there are conditions that ensure the well-posedness of the Schrödinger equation, that is, conditions that guarantee that the quantum Hamiltonians are at least essentially self-adjoint operators on their respective domains and lead to unitary evolution. These conditions are directly related with the infinite-dimensional nature of the problem. Second, there are connectedness conditions. These are conditions on the type of interactions considered, the control fields, which need to ensure that the controlled dynamics has no invariant subspaces. This second category is the one most directly related with the controllability properties of finite-dimensional systems. Third, there are non-resonance conditions. These are conditions on the sequence of eigenvalues of the drift operator, the Hamiltonian responsible of the free (or uncontrolled) dynamics. This last set of conditions is not strict and can be relaxed for concrete applications \cite{BoscainCaponigroChambrionEtAl2012, ChambrionMasonSigalottiEtAl2009, balmaseda2022global}. This is also the case in the concrete example that we consider in Section~\ref{sec:ExampleDiracDelta}.

By considering control problems formulated as form perturbations we can easily characterise the conditions about the well-posedness of the problem. This formulation is very general and can be used in a wide variety of physical situations of interest, e.g.\ \cite{Simon1971, post2016boundary, IbortLledoPerezPardo2015b, BalmasedaLonigroPerezPardo2023b, IbortLlavonaLledoEtAl2021} as well as being suited for numerics \cite{LopezYelaPerezPardo2017, IbortPerezPardo2013}. We rely on this formulation and recent stability results on the solutions of the Schrödinger equation, cf.\ \cite{BalmasedaLonigroPerezPardo2023}, to achieve global approximate controllability results by extending Lie-Galerkin results to form bilinear systems. In particular, we prove the existence of \emph{a priori} bounds on the $L^1$-norm of the control functions for generic initial and target states and we present the analogues of the connectedness and spectral non-resonant conditions. As already mentioned, the connectedness conditions are necessary, while the non-resonant conditions can be relaxed in the concrete applications. In Section~\ref{sec:results} we introduce the main results of this article and present families of physical problems that can be addressed by this formalism and that could not be handled with previous results in the literature about quantum control. More concretely, we devote Section~\ref{sec:ExampleDiracDelta} to treat the case of a quantum particle in a one-dimensional box controlled by a point-like interaction, also called Dirac delta interaction, with varying strength and prove that it is indeed approximately controllable. Point-like interactions are a special type of interactions that can be used to model impurities, perturbations of very short range, or interphases \cite{AlbeverioGesztesyHoeghKrohnEtAl2005, exner1989electrons, ExnerSeba1987, brasche1994schrodinger, bruning2004spectral, albeverio2013remarkable, albeverio2015hamiltonian}.

This article is organised as follows. In Section~\ref{sec:results} we present an overview of the main results together with relevant definitions. In Section~\ref{sec:dynamics} we discuss the dynamics associated with form bilinear control systems and their stability properties. In Section~\ref{sec:chambrion} we refine previous results about the controllability of bilinear control systems by $L^1$-bounded control functions and in Section~\ref{sec:proof} we complete the proofs of the theorems presented in Section~\ref{sec:results}. Finally in Section~\ref{sec:ExampleDiracDelta} we use the results obtained to prove a controllability result for a free quantum particle in a one-dimensional box controlled by a point-like interaction with variable strength.

\section{Preliminaries and main results}\label{sec:results}\nopagebreak[4]

\nopagebreak[4]\subsection{Basic notions}
Let $\hilb$ be a complex, separable, infinite-dimensional Hilbert space, with inner product $\langle \cdot, \cdot \rangle$ antilinear on its first argument.
We will denote the associated norm by $\|\cdot\|$.
We shall consider a nonnegative\footnote{We have restricted ourselves to the case of nonnegative operators, but the extension to operators that are semibounded from below is straightforward, as done in \cite{BalmasedaLonigroPerezPardo2023}. Indeed, in this case there is $m>0$ such that $H_0+m>0$ and adding this constant does not affect the domains of the operators and leads to propagators that are equal up to a phase factor.} self-adjoint operator $H_0$, densely defined in $\hilb$ with domain $\dom H_0\subset\hilb$.
It is known, cf.\ Theorem~\ref{thm:representation_theorem},  that $H_0$ is uniquely associated with a closed Hermitian sesquilinear form 
\begin{equation*}
  h_0:\,\hilb^+\times\hilb^+\rightarrow\mathbb{C},
\end{equation*}
with $\hilb^+\supset\dom H_0$ being the \textit{form domain} of $H_0$, characterised by the equality $h_0(\Psi,\Phi)=\Braket{\Psi,H_0\Phi}$ for all $\Psi\in\hilb^+$ and $\Phi\in\dom H_0$. Here, closed means that $\hilb^+$, endowed with the norm
\begin{equation}\label{eq:plusnorm0}
  \|\Psi\|_+:=\sqrt{h_0(\Psi,\Psi)+\|\Psi\|^2},
\end{equation}
is a Hilbert space. For a comprehensive introduction to the notion of closed quadratic forms and its relation with self-adjoint operators we refer to \cite[Chapter VI]{Kato1995} and \cite[Chapter 4]{davies1995spectral}. 
It is customary to associate $H_0$ with a \textit{scale of Hilbert spaces} $\hilb^+\subset\hilb\subset\hilb^-$, where $\hilb^-$ is the continuous dual of $\hilb^+$ with respect to the topology induced by $\|\cdot\|_+$. We refer to Section~\ref{subsec:dynamics} for details. In our scenario, the operator $H_0$ plays the role of the \textit{drift} Hamiltonian of a quantum system, i.e., it governs the free (or uncontrolled) evolution of the system.

\begin{definition}\label{def:relatively_bounded_form}
  Let $\hilb^+$ be a dense subset of $\hilb$, and	$h_0:\,\hilb^+\times\hilb^+\rightarrow\mathbb{C}$ a nonnegative Hermitian sesquilinear form. A sesquilinear form $h_1:\hilb^+\times\hilb^+$ is said to be \textit{relatively bounded} with respect to $h_0$ (or $h_0$-\textit{bounded}) if there exists $a>0$, $b\geq0$ such that
  \begin{equation}\label{eq:relbound}
    |h_1(\Phi,\Phi)|\leq a\,h_0(\Phi,\Phi)+b\|\Phi\|^2\qquad\text{for all }\Phi\in\hilb^+.
  \end{equation}
  The constant $a$ is called the relative bound. In particular, $h_1$ is \textit{infinitesimally} relatively bounded with respect to $h_0$ (or infinitesimally $h_0$-bounded) if for any $\varepsilon>0$ there exists $b(\varepsilon)\geq0$ such that
  \begin{equation*}
    |h_1(\Phi,\Phi)|\leq \varepsilon\,h_0(\Phi,\Phi)+b(\varepsilon)\|\Phi\|^2\qquad\text{for all }\Phi\in\hilb^+.
  \end{equation*}
\end{definition}

\begin{remark}\label{rem:singular}
  By  the representation theorem for sesquilinear forms (cf.~Theorem~\ref{thm:representation_theorem}) and the KLMN Theorem (cf.~Theorem~\ref{thm:klmn}), for all $0\leq u<1/a$ the form $h_0+uh_1$ is uniquely associated with a self-adjoint operator $H_u$ on $\hilb$, with generally $u$-dependent domain $\dom H_u$ but $u$-independent form domain $\hilb^+$. Such an operator does \textit{not} admit, in general, a representation in the form $H_0+uH_1$, since:
  \begin{itemize}
    \item in general, the form $h_1$ is not assumed to be associated with some densely defined linear operator $H_1$ on $\hilb$;
    \item even if so, the sum $H_0+uH_1$ will generally be ill-defined on (a core of) $\dom H_0$.
  \end{itemize}
  However, while $h_1$ is generally not associated with an operator $H_1$ on $\hilb$, it can be uniquely associated with a continuous operator $H_1\in\mathcal{B}(\hilb^+,\hilb^-)$, cf.~Section~\ref{subsec:dynamics} for details.
\end{remark}

For our control purposes, we shall need to consider time-dependent values of the parameter $u$ above---that is, time-dependent sesquilinear forms $h_0+u(t)h_1$, $t\mapsto u(t)$ being a real-valued function such that $0\leq u(t)<1/a$. Of course, the regularity of such functions must be taken into account. Let us first introduce some notation and useful definitions. In what follows we shall denote by $H_{u}(t)$, with ${t\in\mathbb{R}}$, also called time-dependent Hamiltonian, the family of self-adjoint operators that are associated with the time-dependent sesquilinear forms $h_0 + u(t)h_1$, cf.\ Definition~\ref{def:hamiltonian_const_form}.

For a given real, bounded interval $I$, we denote by $\mathrm{C}^d(I)$ the space of complex functions $f: I \to \mathbb{C}$ admitting $d$ continuous derivatives, and by $\mathrm{C}^d_{\rm pw}(I)$ the space of piecewise differentiable functions satisfying the following property: there exists a finite partition $I = \bigcup_{i=1}^\nu I_i$ such that the restrictions of $f$ to the interior of each subinterval, $f|_{\dot{I}_i}$, are in $C^d(\dot{I}_i)$.
To any function $f\in\mathrm{C}_{\rm pw}^d(I)$ we can associate a ``piecewise derivative'' $f'$ by setting $f'|_{\dot{I}_i} = \frac{\d}{\d t} (f|_{\dot{I}_i})$ and arbitrarily assigning its values at the extremal points of each interval $I_i$; the choice of the latter values shall be immaterial for our purposes. With this construction, $f' \in \mathrm{C}_{\rm pw}^{d-1}(I)$. For the rest of this section we shall assume that $I\subset \mathbb{R}$ is a bounded interval. 

\begin{definition}[Unitary propagator]\label{def:unitary_propagator}
  A \textit{unitary propagator} is a two-parameter family of unitary operators $U(t,s)$ satisfying:
  \begin{enumerate}[label=\textit{(\roman*)},nosep,leftmargin=*]
    \item $U(t,s)U(s,r)= U(t,r)$.
    \item $U(t,t)= \mathbb{I}$, where $\mathbb{I}$ is the identity operator on $\hilb$.
    \item $(t,s)\in I \times I \mapsto U(t,s)$ is jointly strongly continuous.
  \end{enumerate}
\end{definition}

\begin{definition}[Weak and piecewise weak solution of the Schrödinger equation]\label{def:piecewise_solution}
  Let $H_{u}(t)$, $t\in I$, be defined as above.
  We say that a unitary propagator $U(t,s)$, $t,s\in I$, is 
  \begin{itemize}
    \item a \textit{weak solution of the Schrödinger equation} generated by $H_{u}(t)$ if, for all $\Psi_0\in \hilb^+$, the function $\Psi:t\mapsto U(t,s)\Psi_0$ satisfies		
      \begin{equation}\label{eq:schrodinger_weak}
        i\frac{\mathrm{d}}{\mathrm{d}t}\Braket{\Phi,\Psi(t)}=h_0\left(\Phi,\Psi(t)\right)+u(t)h_1\left(\Phi,\Psi(t)\right)
      \end{equation}
      for all $\Phi\in\hilb^+$ and $t\in I$.
    \item a \textit{piecewise weak solution of the Schrödinger equation} generated by $H_{u}(t)$
      if there exists $t_0<t_1<\ldots<t_d\in I$ and a family of 
      weak
      solutions of the Schrödinger Equation
      $\{U_i(t,s) \mid t,s \in (t_{i-1},t_i)\}_{i=1,\dots,d}$ such that
      \begin{equation*}
        U(t,s)=U_i(t,t_{i-1})U_{i-1}(t_{i-1},t_{i-2})\cdots U_j(t_{j},s).
      \end{equation*}
  \end{itemize}
\end{definition}

It will be shown in Proposition~ \ref{prop:existence0} that, under appropriate conditions, for every $u\in\mathrm{C}^1_{\rm pw}(I)$ such that $0\leq u(t)<1/a$, there exists a piecewise weak solution of the Schrödinger equation generated by the corresponding Hamiltonian $H_{u}(t)$,  $U(t,s)$, $t,s\in I$.

\subsection{Form bilinear control systems}
We are now in the position of addressing whether, for a given choice of initial and target states $\Psi_0,\Psi_1$ of the system, we can select $u(t)$ in such a way to drive the system, initially prepared in the state $\Psi_0$, arbitrarily close to $\Psi_1$. Let us start by setting the basic definitions and assumptions:

\begin{definition} \label{def:form_linear_control_system}
  A \textit{form bilinear quantum control system} is a triple $(H_0,h_1,r)$, where 
  \begin{itemize}
    \item 	$H_0$ is a nonnegative self-adjoint operator with form domain $\hilb^+$; 
    \item      there exists an orthonormal basis $\{\Phi_k\}_{k \in \mathbb{N}}$ of $\hilb$ made out of the eigenvectors of $H_0$;
    \item 	$h_1:\,\hilb^+\times\hilb^+\rightarrow\mathbb{C}$ is a Hermitian sesquilinear form, \textit{relatively bounded} with respect to the form $h_0$ associated with $H_0$, with relative bound $a$, cf.\ Definition \ref{def:relatively_bounded_form};
    \item   $r$ is a real number in $(0,1/a)$.
  \end{itemize}
\end{definition}

The condition on the existence of the orthonormal basis is guaranteed whenever $H_0$ has a compact resolvent, which is the situation in many cases of physical interest, e.g. for the Laplace--Beltrami operator on compact manifolds.

  The constant $r$ represents the maximal possible intensity of the controls to be considered; the limitation $r<1/a$ guarantees self-adjointness.
  If, in particular, $h_1$ is \textit{infinitesimally} relatively bounded with respect to $h_0$, the parameter $r$ can be chosen arbitrarily.

\begin{definition}\label{def:connect_chain}
  Let $(H_0,h_1,r)$ be a form bilinear control system. A subset $S\subset\mathbb{N}^2$ is a \textit{connectedness chain} for $(H_0,h_1,r)$ if for all $(j,\ell)\in\mathbb{N}^2$ there exists a finite sequence
  \begin{equation*}
    (j,s_1),\;(s_1,s_2),\;(s_2,s_3),\;\ldots,\;(s_{k-1},s_k),\;(s_{k},\ell)\in S
  \end{equation*}
  such that
  \begin{equation*}
    h_1(\Phi_j, \Phi_{s_1}) ,\;
    h_1(\Phi_{s_1},\Phi_{s_2}) , \;\ldots,\;
    h_1(\Phi_{s_{k-1}},  \Phi_{s_k}) ,\;
    h_1(\Phi_{s_{k}}, \Phi_{\ell})  \neq 0.
  \end{equation*}
\end{definition}

\begin{remark}\label{rem:connect_chain}
  Definition~\ref{def:connect_chain} is an adaptation to the form bilinear scenario of the definition of connectedness chain in~\cite{BoscainCaponigroChambrionEtAl2012}, where one has $h_1(\Phi_{j},\Phi_\ell)=\Braket{\Phi_j,H_1\Phi_\ell}$ for some densely defined symmetric operator $H_1$ on $\hilb$. We also point out that the existence of a connectedness chain can be equivalently restated as follows: the \textit{graph} $(V,E)$ with set of vertices $V=\mathbb{N}^2$ and edges $E:=\{(j,\ell)\in\mathbb{N}^2:h_1(\Phi_j,\Phi_\ell)\neq0\}$ is connected. 

  Note that it may happen that a connectedness chain for a subset of the eigenvectors $\{\Phi_j\}_{j\in\mathbb{N}}$, which is however not a connectedness chain for all eigenvectors, is found.
  In such a case, the discussion below still applies by restricting the Hilbert space $\hilb$ to the space spanned by the ``connected'' eigenvectors---that is, only taking into account initial and target states belonging to such a space. This means that one can concentrate on the connected components of the graph $(V,E)$.
\end{remark}

\begin{assumption} \label{assump:form}
  Let $(H_0, h_1,r)$ be a form bilinear control system. Let $\lambda_k$ be the eigenvalue associated to the eigenfunction $\Phi_k$, $k\in\mathbb{N}$. We assume that:
  \begin{enumerate}[label=\textit{(A\arabic*)},nosep,leftmargin=*]
    \item\label{ass:A2} the system admits a \textit{non-resonant connectedness chain} $S\subset\mathbb{N}^2$, that is, a connectedness chain (cf.\ Definition \ref{def:connect_chain}) such that, in addition, the following holds: for all $(s_1,s_2)\in S$ and for all $(t_1,t_2)\in\mathbb{N}^2 \setminus \{(s_1, s_2), (s_2, s_1)\}$ such that $h_1(\Phi_{t_1},\Phi_{t_2})\neq0$, it holds $|\lambda_{s_2}-\lambda_{s_1}|\neq|\lambda_{t_2}-\lambda_{t_1}|$;
  \item\label{ass:A3} if $\lambda_j=\lambda_k$ for $j\neq k$, then $h_1(\Phi_j,\Phi_k)=0$.
  \end{enumerate}
\end{assumption}
\begin{remark}
 Conditions \ref{ass:A2} and \ref{ass:A3} are technical conditions on the sequence of eigenvalues needed to prove the controllability. These are sufficient but not necessary conditions and can be relaxed in many cases of interest \cite{ChambrionMasonSigalottiEtAl2009, BoscainCaponigroChambrionEtAl2012, balmaseda2022global}. 
  
\end{remark}
The discussion in Remark~\ref{rem:singular} crucially differentiates form bilinear systems from the one usually considered in the literature, which we will refer to as (\textit{operator}) bilinear control systems (cf.~Definition~\ref{def:bilinear_control_system}), in which $H_1$ is a legitimate operator on $\hilb$ such that $H_0+uH_1$ is at least essentially self-adjoint on a core of $\dom H_0$ (e.g. $H_1$ is relatively bounded, in the operator sense, with respect to $H_0$). This prevents a direct application of the available results in the literature to our setting, since the proofs of such results heavily rely on the fact that the domain of $H_1$ contains a complete set of eigenvectors $\{\Phi_j\}_{j\in\mathbb{N}}$ of $H_0$, cf.~Section~\ref{sec:chambrion} and Assumption~\ref{assump:boscain}. 
The basic idea is to extend those controllability results to form bilinear control systems whenever the latter can be approximated, in a suitable topology, by ``regular" ones to which the aforementioned results apply. We need the following notions of approximation. 
\begin{definition}\label{def:approx_family}

  Let $(H_0,h_1,r)$ be a form bilinear system and  let $\{h_1^{(n)}\}_{n\in\mathbb{N}}$ be a sequence of $h_0$-bounded sesquilinar forms. We say that $\{h_1^{(n)}\}_{n\in\mathbb{N}}$ is an \textit{approximating family} for $(H_0,h_1,r)$ if $h_1^{(n)}\neq h_1$ for every $n$ and 
  \begin{equation}\label{eq:convergence2}
    \adjustlimits\lim_{n\to\infty}\sup_{\Psi, \Phi \in \hilb^+} \frac{\bigl|h_1(\Psi, \Phi)-h_1^{(n)}(\Psi,\Phi)\bigr|}{\|\Psi\|_+ \|\Phi\|_+}=0.
  \end{equation}
  Besides, it is a \textit{regular approximating family} if in addition there exists a family $\{H_1^{(n)}\}_{n\in\mathbb{N}}$ of symmetric linear operators on $\hilb$, with domain $\dom H_1^{(n)}\subset\hilb^+$, such that
  \begin{itemize}
    \item[(i)] $h_1^{(n)}(\Psi,\Phi)=\braket{\Psi,H_1^{(n)}\Phi}$ for all $\Psi\in\hilb^+,\;\Phi\in\dom H_1^{(n)}$;
    \item[(ii)] $\Phi_j\in\dom H_1^{(n)}$ for every $j,n\in\mathbb{N}$, and $H_0+uH_1^{(n)}:\operatorname{span}\left\{\Phi_j\colon j\in\mathbb{N}\right\}\rightarrow\hilb$ is essentially self-adjoint for every $u\in[0,r)$.
  \end{itemize}
\end{definition}

  The request (ii) above ensures that the triple $(H_0,H_1^{(n)},r)$ is a \textit{bilinear control system} in the ``usual'' (operator) sense, cf.\ Definition \ref{def:bilinear_control_system}.
  Notice that, if (ii) is only ensured for $u\in[0,r')$ with $0<r'<r$, one can simply redefine the parameter $r$ choosing it small enough so that (ii) is guaranteed.
\begin{remark}
  As a particular case, if there exists a family of \textit{bounded} operators $\{H_1^{(n)}\}_{n\in\mathbb{N}}$ such that
  \begin{equation*}
     \adjustlimits\lim_{n\to\infty}\sup_{\Psi, \Phi \in \hilb^+} \frac{\bigl|h_1(\Psi, \Phi)-\braket{\Psi,H_1^{(n)}\Phi}\bigr|}{\|\Psi\|_+ \|\Phi\|_+}=0,
  \end{equation*}
  then the operators $\{H_1^{(n)}\}_{n\in\mathbb{N}}$ define a regular approximating family for $(H_0,h_1,r)$ by simply setting $h_1^{(n)}(\Psi,\Phi):=\braket{\Psi,H_1^{(n)}\Phi}$. This will be the case for the subclass of form bilinear control systems addressed in Theorem~\ref{thm:compact}. 

  A slightly more general case in which (ii) is automatically satisfied is the one in which the operators $H_1^{(n)}$ are \textit{relatively bounded} with respect to $H_0$, in which case, by the Kato--Rellich theorem, $H_0+uH_1^{(n)}$ with domain $\dom H_0$ is automatically self-adjoint for small enough $u$. This means that we are considering the following situation: each $H_1^{(n)}$ is $H_0$-bounded {and} the corresponding sesquilinear form $h_1^{(n)}$ is $h_0$-bounded.
\end{remark}

\subsection{Main results}
We can now present our controllability results.
Without loss of generality, we shall fix $s=0$ as the initial evolution time.
The first controllability result concerns the possibility of driving a form bilinear system between linear combinations of eigenvectors of $H_0$:
\begin{theorem} \label{thm:main1}
  Let $(H_0, h_1,r)$ be a form bilinear control system satisfying Assumption~\ref{assump:form} and admitting a regular approximating family (cf.~Definition~\ref{def:approx_family}). For every $m\in\mathbb{N}$, $\varepsilon > 0$, $r > 0$, and $\Psi_0, \Psi_1 \in \lspan{\Phi_j \mid 1 \leq j \leq m}$ with $\|\Psi_0\| = \|\Psi_1\|$, there exists a time $T_u>0$ and a piecewise constant control $u: [0, T_u] \to [0, r)$ satisfying
  \begin{equation*}
    \|u\|_{L_1} \leq \frac{5(m - 1)\pi}{2 \min\{| h_1(\Phi_j,\Phi_k) | \colon (j,k) \in S,\, 1 \leq j,k \leq m\}}
  \end{equation*}
  and such that
  \begin{equation*}
    \left\|\Psi_1-U_u(T_u,0)\Psi_0\right\|<\varepsilon,
  \end{equation*}
  where $U_u(t,s)$, $t,s\in[0,T_u]$, is the piecewise weak solution of the Schrödinger equation generated by the time-dependent Hamiltonian $H_{u}(t)$, cf.~Definition~\ref{def:piecewise_solution}.
\end{theorem}

The second controllability result extends Theorem~\ref{thm:main1} to arbitrary initial and target states:

\begin{theorem}\label{thm:main2}
  Let $(H_0, h_1,r)$ be a form bilinear control system satisfying Assumption~\ref{assump:form} and admitting a regular approximating family (cf.~Definition~\ref{def:approx_family}). 
  Then for every $\Psi_0, \Psi_1 \in \hilb$ with $\|\Psi_0\| = \|\Psi_1\|$ and every $\varepsilon > 0$, there exists a natural number $m = m(\varepsilon, \Psi_0, \Psi_1)$, a time $T_u$ and a piecewise constant control $u: [0, T_u] \to [0, r)$ satisfying
  \begin{equation*}
    \|u\|_{L_1} \leq \frac{5(m - 1)\pi}{2 \min\{| h_1(\Phi_j, \Phi_k) | \colon (j,k) \in S,\, 1 \leq j,k \leq m\}}.
  \end{equation*}
  and such that
  \begin{equation*}
    \left\|\Psi_1-U_u(T_u,0)\Psi_0\right\|<\varepsilon,
  \end{equation*}
  with $U_u(t,s)$ as in Theorem~\ref{thm:main1}.
\end{theorem}
Finally, the existence of a regular approximating family turns out to be guaranteed whenever the perturbation $h_1$ is represented by a continuous operator $H_1\in\mathcal{B}(\hilb^+,\hilb^-)$ that is \textit{compact}; cf.~Section~\ref{subsec:examples}, leading us to two immediate corollaries of Theorems~\ref{thm:main1}--\ref{thm:main2}. We shall state explicitly the latter:

\begin{theorem}\label{thm:compact}
  Let $(H_0,h_1,r)$ be a form bilinear control system satisfying Assumption~\ref{assump:form} and such that $H_1\in\mathcal{B}(\hilb^+,\hilb^-)$ is compact. Then, for every $\Psi_0, \Psi_1 \in \hilb$ with $\|\Psi_0\| = \|\Psi_1\|$ and every $\varepsilon > 0$, there exists a natural number $m = m(\varepsilon, \Psi_0, \Psi_1)$, a time $T_u$ and a piecewise linear control $u: [0, T_u] \to [0, r)$ satisfying
  \begin{equation*}
    \|u\|_{L_1} \leq \frac{5(m - 1)\pi}{2 \min\{| h_1(\Phi_j, \Phi_k) | \colon (j,k) \in S,\, 1 \leq j,k \leq m\}}.
  \end{equation*}
  and such that
  \begin{equation*}
    \left\|\Psi_1-U_u(T_u,0)\Psi_0\right\|<\varepsilon,
  \end{equation*}
  with $U_u(t,s)$ as in Theorem~\ref{thm:main1}.
\end{theorem}

We will prove Theorems~\ref{thm:main1}--\ref{thm:compact} in Section~\ref{sec:proof}.

  The cases covered by Theorem~\ref{thm:compact} include, among others, form bounded singular perturbations of self-adjoint operators---that is, operators formally corresponding to expressions like
  \begin{equation*}
    H_0+u(t)\sum_{j,\ell=1}^N v_{j\ell}(f_j,\cdot)f_\ell,
  \end{equation*}
  where $(f,\cdot)$, $f\in\hilb^-$, represents a continuous linear functional on $\hilb^+$.

  A typical example of form bounded singular perturbation is the following situation. Let $\hilb=L^2([0,1])$ and $H_0:=-\Delta_{\rm Dir}+V$, where $\Delta_{\rm Dir}$ is the Laplace operator with Dirichlet boundary conditions and 
  $V$ is a multiplication operator associated with a smooth potential such that $H_0$ is nonnegative. This operator has form domain $\hilb^+:=\mathrm{H}^1_0([0,1])=\{\Phi\in\mathrm{H}^1([0,1])\,\colon\,\Phi(0)=\Phi(1)=0\}$, where $\mathrm{H}^1([0,1])$ is the first order Sobolev space on the unit interval. Since the latter  is compactly embedded in $L^2([0,1])$ this guarantees that $H_0$ has compact resolvent.
  One may consider a control system formally obtained by a \textit{Dirac delta perturbation} of $H_0$ at the position $x_0\in(0,1)$:
  \begin{equation*}
    H_{u}(t)=-\Delta_{\rm Dir}+V(x)+u(t)\delta(x-x_0).
  \end{equation*}
  The expression above, while obviously formal, is properly defined as a rank-one singular perturbation of $H_0$, since the form
  \begin{equation*}
    h_1:\hilb^+\times\hilb^+\rightarrow\mathbb{C},\qquad h_1(\Psi,\Phi)=\overline{\Psi(x_0)}\Phi(x_0)
  \end{equation*}
  is relatively form bounded with respect to the form associated with $-\Delta_{\rm Dir}$, and thus with the form $h_0$ associated with $H_0$. As such, Theorem \ref{thm:compact} provides sufficient conditions for the controllability of such systems via piecewise constant control functions. This situation is explored in detail in Section~\ref{sec:ExampleDiracDelta}, where we prove that such a system is indeed approximately controllable in the subspace of functions with even parity, cf.\ Theorem~\ref{thm:Diracdeltacontrollable}.

Another relevant example in the literature is provided by Friedrichs models \cite{friedrichs1948perturbation, lee1954some, gadella2011friedrichs, facchi2021spectral, lonigro2022self}. These are operators on $\mathbb{C}\oplus\hilb$ formally defined as
  \begin{equation*}
    \begin{pmatrix}
      \omega_0&\\&H_0
      \end{pmatrix}+u(t)\begin{pmatrix}
              &\braket{f,\cdot}\\f&
    \end{pmatrix},
  \end{equation*}
  again with $f\in\hilb^-$, $\omega_0\in\mathbb{R}$ and $H_0$ typically with compact resolvent. 
  In these cases approximate controllability can be ensured provided that Assumption~\ref{assump:form} is verified. The existence of a non-resonant connectedness chain shall be a matter of further analysis. These results could be also adapted to handle other time-dependent problems with explicit time-dependency in the domains, e.g.\ \cite{BalmasedaDiCosmoPerezPardo2019, IbortMarmoPerezPardo2014a, PerezPardoBarberoLinanIbort2015, DellAntonioFigariTeta2000}.

%%%%%%%%%%%%%%%%%%%%%%%%%%%%%%%%%%%%%%%%%%%%%%%%%%%%%%%%%%%%%
%%%%%%%%%%%%%%%%%%%%%%%%%%%%%%%%%%%%%%%%%%%%%%%%%%%%%%%%%%%%%
%%%%%%%%%%%%%%%%%%%%%%%%%%%%%%%%%%%%%%%%%%%%%%%%%%%%%%%%%%%%%

\section{Dynamics and stability of form bilinear control systems}\label{sec:dynamics}

\subsection{Time-dependent Hamiltonians with constant form domain}

Since the relation between sesquilinear forms and operators on a Hilbert space $\hilb$ occupies a central role for the purposes of the present article, we shall hereby provide the precise statement. This result is originally due to Friedrichs \cite{Friedrichs1934} and  here we present it as stated (in a more general scenario) in \cite[Chapter VI, Theorem 2.1]{Kato1995}.

\begin{theorem}[Representation theorem for Hermitian forms]\label{thm:representation_theorem}
  Let $\hilb^+\subset\hilb$ a dense subspace of $\hilb$, and $h:\hilb^+\times\hilb^+\rightarrow\mathbb{C}$ a Hermitian and closed sesquilinear form bounded from below; that is, $h(\Phi,\Phi)\geq-m\|\Phi\|^2$ for some $m\geq0$. Then there exists a unique self-adjoint operator $H:\dom H\subset\hilb\rightarrow\hilb$ such that
  \begin{itemize}
    \item[(i)] $\dom H\subset\hilb^+$ and $h(\Psi,\Phi)=\Braket{\Psi,H\Phi}$ for all $\Psi\in\hilb^+$ and $\Phi\in\dom H$;
    \item [(ii)] $\dom H$ is a core of $h$;
    \item[(iii)] let $\Phi\in\hilb^+$ and $\tilde{\Phi}\in\hilb$ such that $h(\Psi,\Phi)=\braket{\Psi,\tilde{\Phi}}$ for all $\Psi$ in a core of $h$, then $\Phi\in\dom H$ and $\tilde{\Phi}=H\Phi$.
  \end{itemize}
\end{theorem}

  Notice that, by (i), $H$ is also bounded from below with $H\geq-m$; that is, $\braket{\Phi,H\Phi}\geq-m\|\Phi\|^2$. 
  In fact, a one-to-one correspondence between closed, Hermitian forms semibounded from below, and self-adjoint operators bounded from below can be proven.

In Quantum Mechanics, any pure state of a system is associated, uniquely up to a global phase term, with an element of the Hilbert space $\hilb$ with unit norm. 
The evolution of the system in a time interval $I\subset\mathbb{R}$ is given by a family of self-adjoint operators $\{H(t) \mid t \in I\}$, the (possibly time-dependent) \textit{Hamiltonian}, through the Schrödinger equation ($\hbar=1$):
\begin{equation} \label{eq:schrodinger}
  \frac{\d}{\d t} \Psi(t) = -i H(t) \Psi(t), \qquad \Psi(t) \in \dom H(t),
\end{equation}
with initial condition $\Psi(s)=\Psi_0\in\dom H_0$ for some $s\in I$, where $H_0=H(s)$.

As previously discussed, the existence of unitary propagators, cf.\ Definition \ref{def:unitary_propagator}, solving the Schrödinger equation for time-dependent Hamiltonians, especially when $\dom H(t)$ is not constant, is a nontrivial matter. Nevertheless, general results can be found for the following class of operators:
\begin{definition}[Time-dependent Hamiltonian with constant form domain]\label{def:hamiltonian_const_form}
  Let $I \subset \mathbb{R}$ be a compact interval, let $\hilb^+$ be a dense subspace of $\hilb$, and $H(t)$, ${t \in I}$, a family of self-adjoint operators on $\hilb$, each with dense domain $\dom H(t)$.
  We say that $H(t)$,  ${t \in I}$, is a \emph{time-dependent Hamiltonian with constant form domain} $\hilb^+$ if:
  \begin{enumerate}[label=\textit{(\roman*)},nosep,leftmargin=*]
    \item there exists $m >0$ such that, for every $t\in I$, the inequality $\langle \Phi, H(t)\Phi \rangle \geq -m \|\Phi\|^2$ holds for all $\Phi\in\dom H(t)$;
    \item for every $t \in I$, the domain of the closed Hermitian sesquilinear form $h_t$ associated with $H(t)$ by Theorem \ref{thm:representation_theorem} is $\hilb^+$.
  \end{enumerate}
\end{definition}  
By Theorem \ref{thm:representation_theorem}, $h_t$ is a time-dependent, uniformly bounded from below, closed, sesquilinear form 
with constant domain $\hilb^+$ characterised as the unique form satisfying, for all $\Phi\in\hilb^+$ and $\Psi\in\dom H(t)$,
\begin{equation*}
  h_t(\Phi,\Psi)= \langle \Phi,H(t)\Psi \rangle.
\end{equation*}
It satisfies, for all $\Phi,\Psi\in\hilb^+$,
\begin{equation*}
  h_t(\Psi,\Phi)=\overline{h_t(\Phi,\Psi)},\quad h_t(\Phi,\Phi)\geq-m\|\Phi\|^2,
\end{equation*}
where the overline stands for complex conjugation.

It can be shown that, provided that the function $t\mapsto h_t(\Phi,\Psi)$ is in $C^1(I)$, respectively in $C^2(I)$, the Schrödinger equation admits a unique solution in the weak  topology, respectively strong topology. See for instance \cite[Section 3]{Kisynski1964} and~\cite[Thms. II.23 and II.24]{Simon1971}. We also refer to \cite{BalmasedaLonigroPerezPardo2022} for a comparison between the two approaches. However, in this article we consider a situation in which $t\mapsto h_t(\Phi,\Psi)$ admits finitely many singularities. This requires a weaker notion of solution, namely piecewise weak solutions as we introduced in Definition~\ref{def:piecewise_solution}.

We refer to~\cite{BalmasedaLonigroPerezPardo2023} for the existence and stability properties of piecewise weak solutions of the Schrödinger equation generated by a time-dependent Hamiltonian with constant form domain such that $t\mapsto h_t(\Phi,\Psi)$ is piecewise continuously differentiable. We will now apply these results to the case of form bilinear control systems.

\subsection{Dynamics of form bilinear control systems}\label{subsec:dynamics}

In this section, reprising the nomenclature of Section~\ref{sec:results}, $H_0$ will be a nonnegative self-adjoint operator on $\hilb$ with form domain $\hilb^+$, its associated sesquilinear form being $h_0$. Of course, all considerations that follow immediately generalise to the case in which $H_0\geq -m$ for some $m\geq0$.

  It will be useful to briefly recall the formalism of scales of Hilbert spaces which will be used throughout the article, cf.\ \cite[Section 1.1]{Berezanskii1968} for details.

  Given $H_0\geq0$ self-adjoint we define the following norms:
  \begin{equation*}
    \|\Psi\|_\pm=\left\|(H_0+1)^{\pm 1/2}\Psi\right\|.
  \end{equation*}
  The completion of $\dom H_0$ with respect to either of these norms define respectively the Hilbert spaces 
  \begin{equation*}
    \hilb^\pm:=\overline{\dom H_0}^{\|\cdot\|_\pm}.
  \end{equation*}
  It is immediate to verify that the norm  $\|\cdot\|_+$ coincides with the one induced by $h_0$ as introduced in Section \ref{sec:results} (cf.\ Eq. \eqref{eq:plusnorm0}). Therefore, the space $\hilb^+$ is simply the form domain of $H_0$, i.e. $\mathcal{H}^+ =\operatorname{dom}H_0^{1/2}$, justifying the use of the same symbols. Since $\|\cdot\|_-\leq\|\cdot\|\leq\|\cdot\|_+$, one has $\hilb^+\subset\hilb\subset\hilb^-$. Besides:
  \begin{itemize}
    \item the operators $(H_0+1)^{\pm1}$ can be continuously extended to two isometries respectively from $\hilb^+$ to $\hilb^-$ and vice versa, which, with an abuse of notation, we will denote by the same symbols;
    \item $\hilb^+$ and $\hilb^-$ are {dual} with respect to the pairing $(\cdot,\cdot)_{+,-}$ defined by
      \begin{equation*}
        \Psi\in\hilb^+,\;\Phi\in\hilb^-\mapsto(\Psi,\Phi)_{+,-}:=\Braket{(H_0+1)^{1/2}\Psi,(H_0+1)^{-1/2}\Phi},
      \end{equation*}
      and the pairing is \textit{compatible} with the scalar product on $\hilb$, in the following sense: $(\Psi,\Phi)_{+,-}=\braket{\Psi,\Phi}$ whenever $\Psi\in\hilb^+,\Phi\in\hilb$. Finally, a ``Cauchy--Schwarz inequality'' holds:
      \begin{equation*}
        \left|(\Psi,\Phi)_{+,-}\right|\leq\|\Psi\|_+\|\Phi\|_-,\qquad \Psi\in\hilb^+,\;\Phi\in\hilb^-.
      \end{equation*} 
  \end{itemize}
  We will use the symbols $\|\cdot\|_{+,-}$ (resp. $\|\cdot\|_{-,+}$) to denote the operator norm in $\mathcal{B}(\hilb^+,\hilb^-)$ (resp. $\mathcal{B}(\hilb^-,\hilb^+)$). Besides, we will write $(\Phi,\Psi)_{-,+}:=\overline{(\Psi,\Phi)_{+,-}}$.

  Consider a Hermitian sesquilinear form $h_1$ which is relatively bounded with respect to $h_0$, cf.~Definition~\ref{def:relatively_bounded_form}. While in general $h_1$ will not be associated with some densely defined operator $H_1$ on $\hilb$, it is always possible to represent it via a \textit{continuous} operator $H_1\in\mathcal{B}(\hilb^+,\hilb^-)$. Indeed, by Eq.~\eqref{eq:relbound} one has $|h_1(\Phi,\Phi)|\leq M\|\Phi\|_+^2$, where $M=\max\{a,b\}$; that is, $h_1$ is a \textit{bounded} sesquilinear form on $\hilb^+$. Thus by Riesz's Theorem there exists unique operator $H_1\in\mathcal{B}(\hilb^+,\hilb^-)$ such that $h_1(\Phi,\Psi)=(\Psi,H_1\Phi)_{+,-}$. The converse is also true---given $H_1\in\mathcal{B}(\hilb^+,\hilb^-)$, the form $h_1(\Psi,\Phi):=(\Psi,H_1\Phi)_{+,-}$ satisfies $|h_1(\Phi,\Phi)|\leq M\|\Phi\|^2_+$, with $M:=\|H_1\|_{+,-}$, and thus satisfies Eq.~\eqref{eq:relbound} with $a=b=M$.

  The correspondence $h_1(\cdot,\cdot)\leftrightarrow H_1\in\mathcal{B}(\hilb^+,\hilb^-)$ will be often understood from now on. Notice that
  \begin{equation*}
    \|H_1\|_{+,-}:=\sup_{0\neq\Phi\in\hilb^+}\frac{\|H_1\Phi\|_-}{\|\Phi\|_+}=\sup_{0\neq\Psi,\Phi\in\hilb^+}\frac{|h_1(\Psi,\Phi)|}{\|\Phi\|_+\|\Psi\|_+},\qedhere
  \end{equation*}
  that is, the operator norm of $H_1\in\mathcal{B}(\hilb^+,\hilb^-)$ coincides with the norm of $h_1$ understood as a bounded form on $\hilb^+$.

We shall start by proving that all control problems that will be addressed in this work are indeed well-posed. For convenience, let us start by recalling the well-known KLMN theorem, cf.\ \cite[Chapter VI, Theorem~3.4]{Kato1995}:
\begin{theorem}[KLMN Theorem]\label{thm:klmn}
  Let $H_0\geq0$ a self-adjoint operator on $\hilb$, with $h_0$ its associated sesquilinear form, and $h_1$ a Hermitian form relatively bounded with respect to $h_0$ with constants $a,b>0$, cf.\ Definition~\ref{def:relatively_bounded_form}.
  If the relative bound is such that $a<1$, then
  \begin{enumerate}[label=\textit{(\roman*)}]
    \item\label{item:KLMNi} the Hermitian form $h_0+h_1:\hilb^+\times\hilb^+\rightarrow\mathbb{C}$ is bounded from below, with $h_0+h_1\geq -b$;
    \item\label{item:KLMNii} the norm associated with $h_0+h_1$ by
      \begin{equation*}
        \Phi\in\hilb^+\mapsto \sqrt{h_0(\Phi,\Phi)+h_1(\Phi,\Phi)+(b+1)\|\Phi\|^2}
      \end{equation*}
      is equivalent to $\|\cdot\|_+$ with a constant depending only on $a$ and $b$.
  \end{enumerate}
  That is, $h_0+h_1$ is closed and bounded from below.
\end{theorem}
The last statement immediately implies, by Theorem \ref{thm:representation_theorem}, that $h_0+h_1$ is uniquely associated with a self-adjoint operator bounded from below with form domain $\hilb^+$.
\begin{proposition}\label{prop:existence0}
  Let $(H_0,h_1,r)$ be a form bilinear control system on $\hilb$ (cf.~Definition~\ref{def:form_linear_control_system}), $I\subset\mathbb{R}$, and let $u:I\rightarrow(-r,r)$ be a function such that $u\in\mathrm{C}^1_{\rm pw}(I)$. Then 
  \begin{enumerate}[label=\textit{(\roman*)}]
    \item\label{item:propi} the time-dependent operator $H_{u}(t)$ associated with the form $h_0+u(t)h_1$ is a time-dependent Hamiltonian with constant form domain $\hilb^+$;
    \item\label{item:propii} there exists a unique unitary propagator $U(t,s)$, $t,s\in I$, such that
      \begin{itemize}
        \item[(a)] $U(t,s)\hilb^+=\hilb^+$ for all $t,s\in I$;
        \item[(b)] $U(t,s)$ is a piecewise weak solution of the Schrödinger equation, cf.~Definition~\ref{def:piecewise_solution}.
      \end{itemize}
  \end{enumerate}
\end{proposition}

\begin{proof}
  By definition, the form $h_1$ is relatively bounded with respect to the form $h_0$ associated with $H_0$ with bound $a < 1/r$.
  Consequently, since $|u(t)| < r$, the sesquilinear form $u(t)h_1$ is also $h_0$-bounded for every $t\in I$:
  for any $\Phi\in\hilb^+$
  \begin{equation*}
    \left|u(t)h_1(\Phi,\Phi)\right|<r\left|h_1(\Phi,\Phi)\right|\leq ra\,h_0(\Phi,\Phi)+rb\|\Phi\|^2.
  \end{equation*}
  Since $ra<1$, by Theorem~\ref{thm:klmn} the Hermitian sesquilinear form $h_0+u(t)h_1$ is closed for each $t\in I$, and is thus uniquely associated with a self-adjoint operator $H_{u}(t)$ bounded from below, with form domain $\hilb^+$.
  Therefore, $H_{u}(t)$, $t\in I$, is a time-dependent Hamiltonian with constant form domain, proving \ref{item:propi}. As such, by~\cite[Coroll. 3.2]{BalmasedaLonigroPerezPardo2023}, \ref{item:propii} holds. 
\end{proof}

\subsection{Approximations of form bilinear systems}
We shall now consider the situation in which the form bilinear system admits an approximating family $\{h_1^{(n)}\}_{n\in\mathbb{N}}$, cf.\ Definition \ref{def:approx_family}. Let us start with an easy property:

\begin{proposition}\label{prop:approximations}
  Let $(H_0,h_1,r)$ be a form bilinear control system on $\hilb$, $I\subset\mathbb{R}$, and consider a family of sesquilinear forms $\{h_{1}^{(n)}\}_{n\in\mathbb{N}}$, $h_{1}^{(n)}:\hilb^+\times\hilb^+\rightarrow\mathbb{C}$ such that
  \begin{equation}\label{eq:convergence}
    \adjustlimits\lim_{n\to+\infty}\sup_{0\neq\Psi, \Phi \in \hilb^+} \frac{\bigl|h_1(\Psi, \Phi)-h_1^{(n)}(\Psi,\Phi)\bigr|}{\|\Psi\|_+ \|\Phi\|_+}=0.
  \end{equation}
  Then there is $n_0\in\mathbb{N}$ and $\tilde{a},\tilde{b}\geq0$, with $r\tilde{a}<1$, such that, for all $n\geq n_0$, 
  \begin{equation*}
    \left|h_{1}^{(n)}(\Phi,\Phi)\right|\leq \tilde{a}h_0(\Phi)+\tilde{b}\|\Phi\|^2,
  \end{equation*}	
  and $(H_0,h_1^{(n)},r)$ is a form bilinear system.
\end{proposition}	
\begin{proof}
  By Eq.~\eqref{eq:convergence}, given any $\mu>0$ there is $n_0\in\mathbb{N}$ such that, for all $n\geq n_0$ and $\Phi\in\hilb^+$,
  \begin{align*}
    \left|h_{1}^{(n)}(\Phi,\Phi)\right|&\leq	\left|h_{1}(\Phi,\Phi)\right|+	\left|h_{1}^{(n)}(\Phi,\Phi)-h_{1}(\Phi,\Phi)\right|\nonumber\\
                                       &\leq a\,h_0(\Phi,\Phi)+b\|\Phi\|^2+\mu\|\Phi\|_+^2\nonumber\\
                                       &=(a+\mu)h_0(\Phi,\Phi)+(b+\mu)\|\Phi\|^2,
  \end{align*}
  and in particular, since $ra<1$, we can choose $\mu$ small enough so that $r(a+\mu)<1$. With this choice, the claim holds with $\tilde{a}:=a+\mu$ and $\tilde{b}:=b+\mu$.
\end{proof}

A consequence of this proposition is that the requirement of the relative boundedness of the family of forms in the Definition~\ref{def:approx_family} is unnecessary as the limit in Eq.~\eqref{eq:convergence} guarantees the uniform relative boundedness of the approximating family. Notice that the condition $r\tilde{a}<1$ is satisfied only if $n$ is large enough.

  As a consequence of Proposotion~\ref{prop:approximations}, for large enough $n$ the triple $(H_0,h_1^{(n)},r)$ is itself a form bilinear control system. To simplify the notation, from now on we shall always assume  without loss of generality that, given an approximating family $\{h_1^{(n)}\}_{n\in\mathbb{N}}$, $(H_0,h_1^{(n)},r)$ is a form bilinear control system for \textit{every} $n\in\mathbb{N}$. 
  
  For the rest of this section we will denote by $H_{u}^{(n)}(t)$ the time-dependent Hamiltonian with constant form domain associated with the sesquilinear form $h_0 + u(t)h_1^{\smash{(n)}}$, and by $U_u^{\smash{(n)}}(t,s)$ the solution of the corresponding Schrödinger equation in case that it exists.

To conclude this section, let us analyse the relation between the controllability properties of $(H_0,h_1,r)$ and of its approximations, which will be a crucial step towards the proof of the main results but is also of interest on its own. 

\begin{theorem}\label{thm:approximating_families_controllability}
  Let $(H_0,h_1,r)$ be a form bilinear control system on $\hilb$ admitting an approximating family $\{h_{1}^{(n)}\}_{n\in\mathbb{N}}$, cf.\ Definition \ref{def:approx_family}. Let $\Psi_0,\Psi_1\in\hilb$ with $\|\Psi_0\|=\|\Psi_1\|$. Suppose that for every $\tilde{\varepsilon}>0$ there is $K=K(\Psi_0,\Psi_1)>0$ such that there exists a sequence of positive numbers $\{T_n\}_{n\in\mathbb{N}}\subset\mathbb{R}$ and a sequence of functions $\{u_n\}_{n\in\mathbb{N}}$ with $u_n\in\mathrm{C}^2_{\rm pw}([0,T_n])$ satisfying
  \begin{equation}\label{eq:inequalities}
    0\leq u_n(t)<r,\qquad\|u_n\|_{L^1([0,T_n])}<K,\qquad\|u_n'\|_{L^1([0,T_n])}<K
  \end{equation}
  and
  \begin{equation*}
    \left\|\Psi_1-U^{(n)}_{u_n}(T_n,0)\Psi_0\right\|<\tilde{\varepsilon}.
  \end{equation*}
  Then, for all $\varepsilon>0$ there exists $T>0$ and $u\in\mathrm{C}^2_{\rm pw}([0,T])$ with $0\leq u(t)<r,$ $\|u\|_{L^1([0,T])}<K$, $\|u'\|_{L^1([0,T])}<K$ such that
  \begin{equation*}
    \left\|\Psi_1-U_{u}(T,0)\Psi_0\right\|<\varepsilon.
  \end{equation*}
\end{theorem}
Theorem~\ref{thm:approximating_families_controllability} can be stated as follows.
A sufficient condition to ensure that a form bilinear system $(H_0,h_1,r)$, starting from an initial state $\Psi_0$, can be driven $\varepsilon$-close to a target state $\Psi_1$, is that the very same property holds for each $(H_0,h_1^{(n)},r)$ as long as the control functions of the latter, $\{u_n\}_{n\in\mathbb{N}}$, can be chosen in such a way that $\|u_n\|_{L^1}$ and $\|u'_n\|_{L^1}$ are uniformly bounded.
This property can be used to infer the controllability of a form bilinear system from the controllability of a ``suitably nicer'' family of approximations---and, indeed, it will be used to prove the main results.
A comment is due with respect to the $C^2$-regularity appearing in Theorem~\ref{thm:approximating_families_controllability} and Theorem~\ref{thm:stability}. Both theorems make use of stability theorems proven in \cite{BalmasedaLonigroPerezPardo2023} and this extra regularity requirement is needed in the proof of these theorems.

To prove Theorem~\ref{thm:approximating_families_controllability} we shall start by proving an intermediate result concerning approximate controllability in the weaker norm $\|\cdot\|_-$, cf.~Section~\ref{subsec:dynamics}, and then conclude the proof. To this purpose, we will need to invoke 
a stability property of the dynamics generated by time-dependent Hamiltonians with constant form domain, which is a direct application of \cite[Theorem 3.10]{BalmasedaLonigroPerezPardo2023}.

\begin{theorem} \label{thm:stability}
  Let $(H_0,h_1,r)$ be a form bilinear control system on $\hilb$ admitting an approximating family $\{h_1^{(n)}\}_{n\in\mathbb{N}}$, and let $K > 0$.
  There is a constant $L = L(r, K)$ such that, for every $I\subset\mathbb{R}$ and $f\in C_\mathrm{pw}^2(I)$ satisfying $\|f'\|_{L^1(I)} < K$, the family of piecewise weak solutions, $U_f^{(n)}(t,s)$, of the Schrödinger equations with Hamiltonians $H_{f}^{(n)}(t)$, $n \in \mathbb{N}$, satisfy
  \begin{equation*}
    \|U_f^{(n)}(t,s) - U_f^{(m)}(t,s)\|_{+,-} \leq L \|f\|_{L^1(I)} \|H^{(n)}_1 - H^{(m)}_1\|_{+,-},
  \end{equation*}
  where $H^{(n)}_1\in \mathcal{B}(\hilb^+,\hilb^-)$ is the operator representing the form $h_1^{(n)}$.
\end{theorem}
\begin{proof}
  We will apply \cite[Theorem 3.10]{BalmasedaLonigroPerezPardo2023}; to this purpose, one needs to show that \cite[Assumption~3.4]{BalmasedaLonigroPerezPardo2023} holds.
  The assumption \cite[Assumption~3.4-(A2)]{BalmasedaLonigroPerezPardo2023} is implied by the $C^2$-regularity assumption of the Theorem. Notice that:
  \begin{enumerate}[label=\textit{(\roman*)},nosep,leftmargin=*]
    \item Because of Prop. \ref{prop:approximations}, the two forms $h_0+f(t)h_1$ and $h_0+f(t)h_1^{(n)}$ are relatively bounded with respect to $h_0$ with common constants, whence by the KLMN Theorem~\ref{thm:klmn} they are bounded from below via a common constant, which implies \cite[Assumption~3.4-(A1)]{BalmasedaLonigroPerezPardo2023}, and induce uniformly equivalent norms on $\hilb^+$, which implies \cite[Assumption~3.4-(A3)]{BalmasedaLonigroPerezPardo2023}.
    \item Since $\|f'\|_{L^1(I)} < K$, the quantity
      \begin{equation*}
        M_n:=\|f'\|_{L^1(I)}\max\left\{\|H_1\|_{+,-},\|H_1^{(n)}\|_{+,-}\right\}
      \end{equation*}
      is finite. In fact, since $\|f'\|_{L^1(I)}<K$ and $\|H_1^{(n)}\|_{+,-}\to\|H_1\|_{+,-}$, $M_n$ itself is bounded from above by some constant $M$ independent of $n$, which provides the bound required in \cite[Assumption~3.4-(A3)]{BalmasedaLonigroPerezPardo2023}.
  \end{enumerate}
  These facts ensure that \cite[Assumption 3.4]{BalmasedaLonigroPerezPardo2023} apply and therefore a direct application of \cite[Theorem~3.10]{BalmasedaLonigroPerezPardo2023} concludes the proof. 
\end{proof}

\begin{lemma}\label{lemma:approximating}
  Let $(H_0,h_1,r)$ be a form bilinear control system on $\hilb$ admitting an approximating family $\{h_1^{(n)}\}_{n\in\mathbb{N}}$.
  Given $\Psi_0,\Psi_1\in\hilb$ such that $\|\Psi_0\|=\|\Psi_1\|$, under the same notation and assumptions of Theorem \ref{thm:approximating_families_controllability}, one has that for every $\varepsilon>0$ there exists $T>0$ and $u\in\mathrm{C}^2_{\rm pw}([0,T])$ as in Theorem \ref{thm:approximating_families_controllability} such that
  \begin{equation*}
    \left\|\Psi_1-U_{u}(T,0)\Psi_0\right\|_{-}< \varepsilon.
  \end{equation*}
\end{lemma}

\begin{proof}
  Let $\varepsilon>0$ and $\Psi_0,\Psi_1\in\mathcal{H}$ with $\|\Psi_0\|=\|\Psi_1\|$. By assumption, there exists $K>0$ such that there is a sequence $\{T_n\}_{n\in\mathbb{N}}$, $T_n>0$, and a sequence of functions $\{u_n\}_{n\in\mathbb{N}}$ with $u_n \in\mathrm{C}^2_{\rm pw}([0,T_n])$ satisfying \eqref{eq:inequalities} and such that
  \begin{equation}\label{eq:prooflemma1}
    \left\|\Psi_1-U^{(n)}_{u_n}(T_n,0)\Psi_0\right\|_-\leq\left\|\Psi_1-U^{(n)}_{u_n}(T_n,0)\Psi_0\right\|<\frac{\varepsilon}{2}.
  \end{equation}
  We are going to prove the result first for the particular case $\Psi_0\in\hilb^+$ and later use a density argument to extend to the general case. Assume that $\Psi_0\in\hilb^+$. For every $n$, consider the two unitary propagators $U^{(n)}_{u_n}(t,s)$ and $U_{u_n}(t,s)$ respectively associated with the forms $h_0+u_n(t)h_1^{(n)}$ and $h_0+u_n(t)h_1$.
  Their existence is guaranteed by Proposition~\ref{prop:existence0}.
  By Theorem~\ref{thm:stability}, it holds
  \begin{equation*}
    \left\|U_{u_n}^{(n)}(T_n,0)-U_{u_n}(T_n,0)\right\|_{+,-}\leq L\|u_n\|_{L^1([0,T_n])}\|H_1-H_1^{(n)}\|_{+,-}.
  \end{equation*}
  Since $\|H_1 - H_1^{(n)}\|_{+,-} \to 0$, there is $n_0$ such that, for every $n > n_0$,
  \begin{equation*}
    \|H_1-H_1^{(n)}\|_{+,-}<\frac{\varepsilon}{2KL\|\Psi_0\|_+},
  \end{equation*}
  and, noticing that $\|u_n\|_{L^1([0,T_n])}<K$, one gets
  \begin{equation}\label{eq:prooflemma2}
    \left\|U_{u_n}^{(n)}(T_n,0)\Psi_0-U_{u_n}(T_n,0)\Psi_0\right\|_{-}
      \leq L\|u_n\|_{L^1([0,T_n])}\|H_1-H_1^{(n)}\|_{+,-}\|\Psi_0\|_+< \frac{\varepsilon}{2}.
  \end{equation}
Combining Eqs.~\eqref{eq:prooflemma1} and \eqref{eq:prooflemma2} one gets that for every $n>n_0$
$$ \left\|\Psi_1-U_{u_n}(T_n,0)\Psi_0\right\|_- \leq  \left\|\Psi_1-U^{(n)}_{u_n}(T_n,0)\Psi_0\right\|_- +  \left\|U_{u_n}^{(n)}(T_n,0)\Psi_0-U_{u_n}(T_n,0)\Psi_0\right\|_{-} < \varepsilon.$$

Now, let $\Psi_0\in\hilb$ as in the statement of the Theorem. Let $\tilde{\Psi}_0\in\hilb^+$ such that $\|{\tilde{\Psi}_0}\|=\norm{{\Psi}_1}$ and $\|\Psi_0-\tilde{\Psi}_0\|<\frac{\varepsilon}{2}$. Notice that for every $\Phi\in\hilb$ and $n$ one has that $\norm{U_{u_n}(T_n,0)\Phi}_-\leq \norm{U_{u_n}(T_n,0)\Phi} =\norm{\Phi}$, which follows by the unitarity of $U_{u_n}(T_n,0)$ in $\hilb$. From the previous argument there exists $u$ as required such that 
$$ \left\|\Psi_1-U_{u}(T,0)\tilde{\Psi}_0\right\|_-< \frac{\varepsilon}{2}.$$
Finally we  have that 
$$\left\|\Psi_1-U_{u}(T,0){\Psi}_0\right\|_-\leq \left\|\Psi_1-U_{u}(T,0)\tilde{\Psi}_0\right\|_- + \left\|U_{u}(T,0)(\tilde{\Psi}_0-{\Psi}_0)\right\|_- <\varepsilon,$$
which completes the proof.
\end{proof}

\begin{proof}[Proof of Theorem~\ref{thm:approximating_families_controllability}]
  From Lemma~\ref{lemma:approximating} there exists a sequence $\{T_j\}_{j\in\mathbb{N}}$ and a sequence $\{u_j\}_{j\in\mathbb{N}}$ with $u_j\in\mathrm{C}^2_{\rm pw}([0,T_j])$ such that
  $$\left\|\Psi_1-U_{u_j}(T_j,0){\Psi}_0\right\|_-<\frac{1}{j}.$$
  Let $\Phi\in\hilb$ and $\varepsilon>0$. Then there exists $\Phi_+ \in \mathcal{H}^+$ such that 
   \begin{equation*}
    \|\Phi - \Phi_+\| < \frac{\varepsilon}{4\|\Psi_0\|}.
  \end{equation*}
Now we have that
\begin{equation*}
    |\langle \Phi, \Psi_1 - U_{u_j}(T_j, 0) \Psi_0 \rangle| \leq 2\|\Psi_0\|\|\Phi - \Phi_+\|  + \|\Phi_+\|_+\left\|\Psi_1-U_{u_j}(T_j,0){\Psi}_0\right\|_-<\frac{\varepsilon}{2} + \frac{1}{j}\|\Phi_+\|_+,
  \end{equation*}
  where we have used that $U_{u_j}(T_j, 0)$ is unitary for every $j$ and $\Psi_1-U_{u_j}(T_j,0){\Psi}_0\in\hilb\subset\hilb^-$.
  It follows that for every $j>2\|\Phi_+\|_+/\varepsilon$ the right hand side is smaller than $\varepsilon$ which proves that $\{U_{u_j}(T_j, 0)\Psi_0\}_{j\in\mathbb{N}}$ is a weakly convergent subsequence. By the unitarity of the propagator $U_{u_j}(T_j,0)$ one has that $\|\Psi_1\| = \|\Psi_0\| = \|U_{u_j}(T_j, 0)\Psi_0\|$ for all $j$, which implies strong convergence, i.e.,  $\|\Psi_1 - U_{u_j}(T_j, 0)\Psi_0\| \to 0$.
\end{proof}

\subsection{Compact form perturbations}\label{subsec:examples}

As discussed in Section~\ref{subsec:dynamics}, all forms $h_1:\hilb^+\times\hilb^+\rightarrow\mathbb{C}$ that are relatively bounded with respect to $h_0$ are uniquely associated with a bounded operator $H_1\in\mathcal{B}(\hilb^+,\hilb^-)$. We will briefly analyse the case in which said operator is \textit{compact}. The following simple property, which is nonetheless crucial for the control purposes of this work, holds in such a case.

\begin{proposition}\label{prop:compact}
  Let $H_1\in\mathcal{B}(\hilb^+,\hilb^-)$. Then $H_1$ is compact if and only if there exists a sequence of symmetric finite-rank operators on $\hilb$, $\{H_1^{(n)}\}_{n\in\mathbb{N}}\in\mathcal{B}(\hilb)$, such that $\|H_1^{(n)}-H_1\|_{+,-}\to0$.
\end{proposition}
\begin{proof}
An operator $H_1\in\mathcal{B}(\hilb^+,\hilb^-)$ is compact if and only if there exists a sequence $\{H_1^n\}_{n}$ of finite rank operators in $\mathcal{B}(\hilb^+,\hilb^-)$ such that $\|{H_1-H_1^n}\|_{+,-}\to0$. Therefore, proving the sufficient condition is trivial since $\mathcal{B}(\hilb)\subset\mathcal{B}(\hilb^+,\hilb^-)$. Let  $\{H_1^n\}_{n}\subset\mathcal{B}(\hilb^+,\hilb^-)$. To prove the necessary condition it suffices to show that there exists a sequence $\{\tilde{H}_1^n\}_n$ of finite-rank operators in $\mathcal{B}(\hilb)$ such that $\|H_1^n-\tilde{H}_1^n\|_{+,-}<\frac{1}{n}$. 

Without loss of generality we can assume that $H_1^n = \sum_{j=1}^n(\xi_j,\cdot)_{-,+}\phi_j$ for $\{\xi_j,\phi_j\}_{j=1,\dots,n} \subset\hilb^-$. Take $\{\tilde{\xi}_j\}\subset\hilb$ and $\{\tilde{\phi}_j\}\subset\hilb$ and define 
$$\tilde{H}_1^n := \sum_{j=1}^n(\tilde{\xi}_j,\cdot)_{-,+}\tilde{\phi}_j.$$
A straightforward calculation shows that 
$$\norm{H_1^n-\tilde{H}^n_1}_{+,-} \leq K \max\{\|\xi_j - \tilde{\xi}_j\|_- + \|{\phi}_j - \tilde{\phi}_j\|_-\}_{j=1,\dots,n}$$ 
for some $K$. Since $\hilb$ is dense in $\hilb^-$ the right hand side can be chosen as small as needed, which finishes the proof. 
\end{proof}

  Prop.~\ref{prop:compact} means that, given a sesquilinear form $h_1$ associated with a compact operator $H_1\in\mathcal{B}(\hilb^+,\hilb^-)$, the form bilinear system $(H_0,h_1,r)$ always admits a regular approximating family, cf.~Definition~\ref{def:approx_family}. This is because $H_1$ is the limit of a sequence of finite-rank (thus bounded) operators $\{\tilde{H}_1^{(n)}\}\in\mathcal{B}(\hilb)$, and the operator $H_0+u\tilde{H}_1^{(n)}$ with domain $\dom H_0$ is self-adjoint for any $u$.

\begin{remark}[Finite-rank form bounded perturbations]
In the particular case of finite-rank relatively form bounded perturbations, it is immediate to show that they are infinitesimally relatively bounded. From Proposition~\ref{prop:compact} it follows that for any $\varepsilon>0$ there exists a finite-rank bounded perturbation such that $\|H_1-\tilde{H}_1\|_{+,-}<\varepsilon$. It follows that for any $\Psi\in\hilb^+$
$$
h_1(\Psi,\Psi) \leq |h_1(\Psi,\Psi)  - \tilde{h}_1(\Psi,\Psi)| + \tilde{h}_1(\Psi,\Psi) \leq \varepsilon\norm{\Psi}^2_+ + M\norm{\Phi}^2.
$$
Therefore the triples $(H_0,h_1,r)$ are form bilinear controls systems for any value $r>0$ whenever $h_1$ is a finite-rank relatively bounded perturbation. 
\end{remark}

%%%%%%%%%%%%%%%%%%%%%%%%%%%%%%%%%%%%%%%%%%%%%%%%%%%%%%%%%%
%%%%%%%%%%%%%%%%%%%%%%%%%%%%%%%%%%%%%%%%%%%%%%%%%%%%%%%%%%
%%%%%%%%%%%%%%%%%%%%%%%%%%%%%%%%%%%%%%%%%%%%%%%%%%%%%%%%%%

\section{Controllability of the bilinear Schrödinger equation with \texorpdfstring{$L^1$}{L1}-bounded controls}\label{sec:chambrion}

In this section we shall take a ``step back'' and analyse the controllability of the bilinear Schrödinger equation generated by \textit{operators} in the form $H_0+uH_1$, that is,
\begin{equation}\label{eq:bilinear}
  i\frac{\mathrm{d}}{\mathrm{d}t}\Psi(t)=\left(H_0+u(t)H_1\right)\Psi(t),
\end{equation}
with $H_1$ being a linear symmetric operator on a dense domain of $\hilb$, $u(t)\in[0,r)$, and the triple $(H_0,H_1,r)$ chosen in such a way that $H_0+uH_1$ is at least essentially self-adjoint on the linear span of a complete set of eigenvectors of $H_0$ (cf.~Assumptions~\ref{assump:boscain}). As discussed in Section~\ref{sec:results} (cf.~Remark~\ref{rem:singular}), form perturbations of $H_0$ do \textit{not} generally fall into this class. However, the results from the previous section will allow us to transfer the results in this section to the case of form bilinear systems admitting an approximating family.

The section is organised as follows. In Section~\ref{subsec:bilinear} we set up notation and state Theorem~\ref{thm:bounded_control_span}, which provides an important refinement of existing results in the literature. In Section~\ref{subsec:prelim_chambrion} we state and prove some preliminary results. In Section~\ref{subsec:proof_chambrion} we finally prove Theorem~\ref{thm:bounded_control_span}.

\subsection{A theorem on the bilinear Schrödinger equation}\label{subsec:bilinear}

We will base on the ideas exposed by Boscain et al. in~\cite{BoscainCaponigroChambrionEtAl2012} to show controllability of Eq.~\eqref{eq:bilinear} with \textit{explicit} bounds for the control functions.

\begin{assumption} \label{assump:boscain}
  Let $H_0, H_1$ be two densely defined symmetric operators on $\hilb$, with domains $\dom H_0$, $\dom H_1$, and let $r > 0$.
  We assume that the triple $(H_0,H_1,r)$ satisfies the following:
  \begin{enumerate}[label=\textit{(B\arabic*)},nosep,leftmargin=*]
    \item\label{ass:chambrion1} $H_0$ is a self-adjoint operator, and there exists an orthonormal basis $\{\Phi_j\}_{j\in\mathbb{N}}\subset\hilb$ of eigenvectors of $H_0$, associated with eigenvalues $\{\lambda_j\}_{j\in\mathbb{N}}$;
    \item\label{ass:chambrion2} $\Phi_j \in \dom H_1$ for every $j$;
    \item\label{ass:chambrion3} $H_0 + u H_1: \lspan{\Phi_j \mid j \in \mathbb{N}} \to \hilb$ is essentially self-adjoint for every $u \in [0,r)$.
  \end{enumerate}
\end{assumption}

\begin{definition}\label{def:bilinear_control_system}
  An (operator) \textit{bilinear quantum control system} is a triple $(H_0,H_1,r)$ satisfying Assumption~\ref{assump:boscain}.
\end{definition}

  These assumptions guarantee, for every piecewise constant function $u(t)$, the existence of a unitary propagator $U_u(t,s)$ such that, for every $\Psi_0\in\lspan{\Phi_j \mid j \in \mathbb{N}}$, the function $t\mapsto U_u(t,s)\Psi_0$ solves Eq.~\eqref{eq:bilinear} with initial condition $\Psi(s)=\Psi_0$ for all but finitely many values of $t$. This also implies that the same function satisfies
  \begin{equation*}
    i\frac{\mathrm{d}}{\mathrm{d}t}\Braket{\Phi,\Psi(t)}=\Braket{\Phi,H_0\Psi(t)}+u(t)\Braket{\Phi,H_1\Psi(t)}
  \end{equation*}
  for any $\Phi\in\hilb$ and for all but finitely many values of $t$. This means that $U_u(t,s)$ is a piecewise weak solution of the Schrödinger equation (cf.\ Definition \ref{def:piecewise_solution}).

\begin{remark}
  By the Kato--Rellich theorem, assumptions \ref{ass:chambrion2} and \ref{ass:chambrion3} are guaranteed whenever
  \begin{itemize}
    \item $H_1$ is \textit{relatively bounded} with respect to $H_0$ (or $H_0$-bounded), that is, $\dom H_0\subset\dom H_1$ and there exists $a,b\geq0$ such that $\|H_1\Phi\|\leq a\|H_0\Phi\|+b\|\Phi\|$ for all $\Phi\in\dom H_0$;
    \item $r<1/a$.
  \end{itemize}
  In particular, if $H_1$ if \textit{infinitesimally} $H_0$-bounded (that is, for every $\varepsilon>0$ there is $b_\varepsilon>0$ such that $\|H_1\Phi\|\leq \varepsilon\|H_0\Phi\|+b_\varepsilon\|\Phi\|$ for all $\Phi\in\dom H_0$), then the parameter $r$ can be chosen arbitrarily.

  We point out that, despite the similarities between these requests and those in Definition~\ref{def:form_linear_control_system}, which may be regarded as the ``form counterpart'' of the former, in general there is not a hierarchical relation between them: an $h_0$-bounded sesquilinear form $h_1$ is generally not associated with an $H_0$-bounded operator $H_1$ nor the converse holds unless extra conditions on $H_1$ are assumed, for instance, that $H_1$ itself is self-adjoint, cf. \cite[Remark 1.3]{brown2004relative}.
\end{remark}

It is known~\cite[Theorem 2.6]{BoscainCaponigroChambrionEtAl2012} that a triple $(H_0,H_1,r)$ is approximately controllable whenever
  Assumption~\ref{assump:boscain} and Assumption~\ref{assump:form} are satisfied.
In addition, we will need to obtain \emph{a priori} bounds on the $L^1$-norms of the control $u:[0,T]\to[0,r)$. This is not straightforwardly implied by the given assumption. Even if $|u(t)|<r$, the control time $T$ might be arbitrarily large.

That such a bound exists was anticipated in \cite[Remark~5.9]{BoscainCaponigroChambrionEtAl2012}. In this article we provide an explicit estimation of the bound. This result is a crucial step towards the proof of the main results presented in Section~\ref{sec:results}.

\begin{theorem} \label{thm:bounded_control_span}
  Let the triple $(H_0,H_1,r)$ be an operator bilinear quantum control system that admits a connectedness chain $S$ and let $m\in\mathbb{N}$.
  Then, for every $\varepsilon > 0$ and every  $\Psi_0, \Psi_1 \in \lspan{\Phi_j \mid 1 \leq j \leq m}$ with $\|\Psi_0\| = \|\Psi_1\|$, there exists a time $T_u$ and a piecewise constant control $u: [0, T_u] \to [0, r)$ such that the propagator solving the Schrödinger equation satisfies $\|U_u(T_u, 0)\Psi_0 - \Psi_1\| < \varepsilon$ and
  \begin{equation*}
    \|u\|_{L_1} \leq \frac{5(m - 1)\pi}{4 \min\{|\langle \Phi_j, H_1 \Phi_k \rangle| \colon (j,k) \in S,\, 1 \leq j,k \leq m\}}.
  \end{equation*}
\end{theorem}

We will prove Theorem~\ref{thm:bounded_control_span} in several steps.
We will first state a finite-dimensional result (Proposition~\ref{prop:full_trajectory}) involving an auxiliary control system on $\mathbb{C}^m$, its proof involving a number of intermediate results (Lemma~\ref{lemma:sigmam_trajectories}, Proposition~\ref{prop:trajectory_ij}).
We will then prove Proposition~\ref{prop:bounded_control_span}, involving the bilinear control system of interest, which will finally lead us to the proof of Theorem~\ref{thm:bounded_control_span}.

\subsection{Preliminary results}\label{subsec:prelim_chambrion}

  Given a bilinear quantum control system $(H_0,H_1,r)$ admitting a connectedness chain $S$, for every $m \in \mathbb{N}$ we define the subset
  \begin{equation*}
    S_m \coloneqq \left\{(i,j) \in S \mid 1 \leq i,j \leq m, i \neq j\right\}
  \end{equation*}
  and the auxiliary control system
  \begin{equation} \label{eq:sigma_m}
    X' = \nu |b_{ij}| \left( \mathrm{e}^{i\theta} e_{ij}^{(m)} - \mathrm{e}^{-i\theta} e_{ji}^{(m)} \right) X,
    \qquad X(0) = I,
    \tag{$\Sigma_m$}
  \end{equation}
  where $\theta = \theta(t) \in \mathbb{S}^1$,  $(i,j) = (i(t), j(t)) \in S_m$ are piecewise constant controls, $b_{ij} = \langle \Phi_i, H_1 \Phi_j \rangle$ and $e_{ij}^{(m)}$ is the $m\times m$ matrix with entry $(i,j)$ equal to one and zero otherwise. The parameter $\nu>0$ is a positive parameter.
  
  In \cite[Proposition 4.5]{BoscainCaponigroChambrionEtAl2012} it is shown that, in a system with a connectedness chain, there always exists a relabelling of the orthonormal basis such that $S_m$ is connected for every $m$. We shall assume this relabelling in what follows.

  \begin{proposition} \label{prop:full_trajectory}
    Let $\Psi \in \mathbb{C}^m$ such that $\|\Psi\| = 1$.
    For any fixed $k \leq m$ and $\nu >0$, there is a trajectory $X(t)$ of~\eqref{eq:sigma_m}, a time
    \begin{equation*}
      T < \frac{(m-1)\pi}{2\nu\min_{(i,j) \in S_m}|b_{ij}|}
    \end{equation*}
    and an angle $\varphi$, such that $X(T) \Psi = \mathrm{e}^{i\varphi} e_k$, with $\{e_j\}_j$ the canonical basis of $\mathbb{C}^m$.
  \end{proposition}	
  As anticipated, the proof of this proposition will involve some simpler results.
  We will obtain the full trajectory as a concatenation of trajectories of~\eqref{eq:sigma_m}. 
  Each one with constant values for the controls $(i,j) \in S_m$ and $\theta \in \mathbb{S}^1$.

  \begin{lemma} \label{lemma:sigmam_trajectories}
    Let $\Psi \in \mathbb{C}^m$ and denote its components by $\psi_k = \langle e_k, \Psi \rangle$.
    Define $T_{ij}^\theta = \mathrm{e}^{i\theta} e_{ij}^{(m)} - \mathrm{e}^{-i\theta} e_{ji}^{(m)}$, and consider the solution of \eqref{eq:sigma_m} given by $X_{ij}^\theta(t) = \mathrm{e}^{\nu |b_{ij}| t T_{ij}^\theta}$.
    Then, the components of $\Psi(t) \coloneqq X_{ij}^\theta \big(\frac{t}{\nu |b_{ij}|}\big)\Psi$, $\psi_k(t) = \langle e_k, \Psi(t) \rangle$, satisfy
    \begin{alignat*}{2}
      \psi_k(t) &= \psi_k \quad (\text{for }k \notin \{i,j\}), \\
      \psi_i(t) &= \psi_i \cos(t) + \mathrm{e}^{i\theta} \psi_j \sin(t), \\
      \psi_j(t) &= \psi_j \cos(t) - \mathrm{e}^{i\theta} \psi_i \sin(t).
    \end{alignat*}
  \end{lemma}
  \begin{proof}
    First, note that $T_{ij}^\theta$ satisfies
    \begin{equation*}
      T_{ij} e_n = 0 \qquad \text{if }n \notin \{i,j\},
    \end{equation*}
    and
    \begin{alignat*}{3}
      \big(T_{ij}^\theta\big)^{2k+1} e_i &= (-1)^{k+1} \mathrm{e}^{-i\theta} e_j, \qquad
                                         &\big(T_{ij}^\theta\big)^{2k} e_i = (-1)^k e_i, \qquad k &= 0, 1, 2, \dots \\
      \big(T_{ij}^\theta\big)^{2k+1} e_j &= (-1)^k \mathrm{e}^{i\theta} e_i, \qquad
                                         &\big(T_{ij}^\theta\big)^{2k} e_j = (-1)^k e_j, \qquad k &= 0, 1, 2, \dots
    \end{alignat*}
    Therefore one has
    \begin{equation*}
      \mathrm{e}^{tT_{ij}} e_n = e_n \qquad \text{for}\qquad n \notin \{i,j\}
    \end{equation*}
    and
$$
  \mathrm{e}^{t T_{ij}^\theta} e_i = \sum_{k = 0}^{\infty} \frac{(-1)^k t^{2k}}{(2k)!} e_i - \sum_{k = 0}^{\infty} \frac{(-1)^k t^{2k+1}}{(2k+1)!} \mathrm{e}^{-i\theta} e_j = \cos(t) e_i - \mathrm{e}^{-i\theta} \sin(t) e_j.
$$
    Similarly, $\mathrm{e}^{t T_{ij}^\theta} e_j = \cos(t) e_j + \mathrm{e}^{i\theta} \sin(t) e_i$.
    Applying the definition of $\Psi(t)$, the result follows.
  \end{proof}

  \begin{proposition} \label{prop:trajectory_ij}
    Let $\Psi \in \mathbb{C}^m$, its components denoted by $\psi_k = \langle e_k, \Psi \rangle$, $\nu>0$, and let $(i,j) \in S_m$.
    For $\ell = i, j$, there exists an angle $\theta$ and
    a time $T \leq \frac{\pi}{2\nu|b_{ij}|}$ such that the trajectory $X_{ij}^\theta(t)$ of~\eqref{eq:sigma_m} satisfies
    \begin{equation*}
      \langle e_\ell, X_{ij}^\theta(T) \Psi \rangle = 0,
      \quad \text{and} \quad
      \langle e_k, X_{ij}^\theta(T)\Psi \rangle = \langle e_k, \Psi \rangle
      \quad (k \notin \{i,j\}).
    \end{equation*}
  \end{proposition}
  \begin{proof}
    Denote $\psi_k(t) \coloneqq \big\langle e_k, X_{ij}^\theta\big( \frac{t}{\nu |b_{ij}|} \big)\Psi \big\rangle$.
    We assume $\psi_i \neq 0$, since otherwise the result is trivial.
    By Lemma~\ref{lemma:sigmam_trajectories}, $\psi_k(t) = \psi_k$ is constant for every $k \notin \{i,j\}$, and
    \begin{equation*}
      \psi_i(t) = \psi_i \cos(t) + e^{i\theta} \psi_j \sin(t).
    \end{equation*}
    If $\psi_j = 0$, then $\psi_i\big(\frac{\pi}{2}\big) = 0$.
    Otherwise, let $\theta_i, \theta_j$ be such that $\psi_i = \mathrm{e}^{i\theta_i} |\psi_i|$ and $\psi_j = -\mathrm{e}^{i\theta_j} |\psi_j|$, and let $\theta = \theta_i - \theta_j$.
    Then,
    \begin{alignat*}{2}
      \psi_i(t) &= \mathrm{e}^{i\theta_i} |\psi_i| \cos(t) - \mathrm{e}^{i(\theta_i - \theta_j)} \mathrm{e}^{\theta_j} |\psi_j| \sin(t), \\
                &= \mathrm{e}^{i\theta_i} \big( |\psi_i| \cos(t) - |\psi_j| \sin(t)\big).
    \end{alignat*}
    The term inside the parenthesis is a real valued continuous function changing signs inside the interval $(0, \frac{\pi}{2})$, and therefore there exists $T \in (0, \frac{\pi}{2})$ such that $\psi_i(T) = 0$.

    For the second part, the argument can be repeated taking in this case $\theta_i$, $\theta_j$ such that $\psi_i = \mathrm{e}^{i\theta_i} |\psi_i|$ and $\psi_j = \mathrm{e}^{i\theta_j} |\psi_j|$, and $\theta' = \theta_j - \theta_i$.
  \end{proof}
  
    It will be useful to exploit the interpretation of the connectedness chain (and, in particular, of the subset $S_m$) in terms of connected graphs, cf.\ Defintion~\ref{def:connect_chain} and Remark~\ref{rem:connect_chain}.
  To the subset $S_m$ we associate the graph with vertex set $\{1, 2, \dots, m\}$ and whose set of edges is $S_m$.
  With a slight abuse of notation, we also denote this graph with the symbol $S_m$ whenever it leads to no confusion.

  Let us recall a few basic notions of graph theory used in the following proofs.
  A graph $G=(V,E)$, with $V\subset\mathbb{N}$, is said to be \textit{connected} if, for every pair of vertices $i,j\in V$, there exists at least a \textit{path} joining them---that is, a sequence of edges of the form $(i,k_1), (k_1, k_2), \dots, (k_n, j) \in E$ joining a sequence of vertices which cannot be repeated.
  The number of edges in a path is called the \textit{length} of the path, and one can define a distance between any two vertices $i,j \in V$ in a connected graph as the length of the shortest path joining them, denoted by $d(i,j)$.  A graph such that every pair of vertices is joined by {one and only one} path is called a \textit{tree}.
  An important property for our purposes is that every connected graph $G = (V, E)$ contains at least a \textit{spanning tree}; that is, there is $E' \subset E$ such that the subgraph $(V, E')$ is a tree.
  Given any tree, one can define a vertex to be its \textit{root}, and then define the \textit{depth} of a vertex to be de distance from it to the root.
  Every vertex on a \textit{rooted tree} has a (unique) \textit{parent vertex}, which is the vertex connected to it on the unique path joining it with the root.

  \begin{proof}[Proof of Proposition~\ref{prop:full_trajectory}]
    The trajectory $X(t)$ can be constructed using the trajectories on Proposition~\ref{prop:trajectory_ij}, and we will use the graph structure of $S_m$ to design it.
    We will denote by $\Psi(t) \coloneqq X(t)\Psi$ the evolution induced by such trajectory and by $\psi_\ell(t) = \langle e_\ell, X(t) \Psi \rangle$ its components.

    Let $S_m' = (V, E)$ be a spanning tree of the graph defined by the connectedness chain $S_m$, and for it we will designate as root the node $k$.
    Using $S_m'$, we will find subtrajectories transferring the weight of each component of $\Psi$ to its parent component, and repeat this process until we transfer the whole initial state to $\psi_k$. Remember that we can always assume that the chains $S_m$ are connected for every $m$, see the discussion at the beginning of Section~\ref{subsec:prelim_chambrion}.

    Denote $V_1 \coloneqq \{\ell \in V \setminus \{k\} \mid \psi_\ell \neq 0\}$, and let $i_1 \in V_1$ be a vertex maximising the depth (that is, the distance from the root) in $V_1$.
    Applying Proposition~\ref{prop:trajectory_ij}, with $j_1$ being the parent of $i_1$ in $S_m'$, there is $T_1 \leq \frac{\pi}{2\nu|b_{i_1j_1}|}$ and $X_1(t)$ such that $\psi_i(T_1) = 0$.
    For $t \leq T_1$, we define $X(t) = X_1(t)$.

    Now, starting from $\Psi(T_1)$ we repeat the argument. Denote $V_2 \coloneqq \{\ell \in V \setminus \{k\} \mid \psi_\ell(T_1) \neq 0\}$ and let $i_2\in V_2$ a vertex maximising the distance from the root and $j_2$ its parent. We obtain a new time $T_2 \leq \frac{\pi}{2\nu|b_{i_2j_2}|}$ and a new trajectory $X_2(t)$. Defining $X(t) = X_2(t-T_1)X_1(T_1)$ for $T_1 < t \leq T_2$, we get that $\psi_{i_2}(T_2) = \psi_{i_1}(T_2) = 0$.

    Therefore, iterating this procedure for every vertex in $V\setminus \{k\}$ defines a trajectory $X(t)$ and a time $T = \sum_{\ell=1}^{p} T_\ell$, with $p<|V\setminus \{k\}|$ such that $\psi_\ell(T) = 0$ for every $\ell \neq k$.
    Since $|V\setminus \{k\}| = m-1$, it holds
    \begin{equation*}
      T = \sum_{\ell=1}^{p} T_\ell \leq \sum_{\ell=1}^{m-1} \frac{\pi}{2\nu \min_{(i,j) \in S_m'} |b_{ij}|} \leq \frac{\pi(m - 1)}{2\nu \min_{(i,j) \in S_m} |b_{ij}|},
    \end{equation*}
    where we have used that any edge in $S_m'$ is in $S_m$.
    This, taking into account the unitarity of $X(t)$, finishes the proof.
  \end{proof}

  \begin{proposition} \label{prop:bounded_control_span}
    Let $H_0, H_1$ satisfy Assumption~\ref{assump:boscain} and let $H(t) = H_0 + u H_1$ admit a non-resonant connectedness chain $S$.
    Let $m > 0$ be fixed.
    Then, for every $\varepsilon > 0$, $r > 0$, $j\in\{1,\dots,m\}$ and $\Psi \in \lspan{\Phi_j \mid 1 \leq j \leq m}$ with $\|\Psi\| = 1$ there exists a piecewise constant control $u: [0, \tau] \to [0, r)$ and $\varphi \in \mathbb{S}^1$ such that the propagator $U(t,s)$ solving the Schrödinger equation for $H(t)$ satisfies
    \begin{equation*}
      \left\|U(\tau, 0)\Psi - \mathrm{e}^{i\varphi} \Phi_j \right\| < \varepsilon
    \end{equation*}
    and
    \begin{equation*}
      \|u\|_{L_1(0,\tau)} \leq \frac{5(m - 1)\pi}{4 \min\{|\langle \Phi_j, H_1 \Phi_k \rangle| \colon (j,k) \in S,\, 1 \leq j,k \leq m\}}.
    \end{equation*}
  \end{proposition}

  \begin{proof}
    Let $m\in\mathbb{N}$ and consider $\Psi\in\mathrm{span}\{\Phi_j|1\leq j\leq m\}$ with $\|\Psi\|=1$. 
    By \cite[Proposition~4.2]{BoscainCaponigroChambrionEtAl2012} the existence of an $L^1$-bound for the controls is equivalent to the existence of an upper bound to the control time in a time-reparametrised system. We will first obtain such an upper bound. 
    By Proposition~\ref{prop:full_trajectory}, there exists a trajectory $X(t)$ of~\eqref{eq:sigma_m} with $X(0)=I_m$, a time
    \begin{equation}\label{eq:Time_bound}
      T < \frac{(m - 1)\pi}{2\nu\min\{|\langle \Phi_j, H_1 \Phi_k \rangle| \colon (j,k) \in S,\, 1 \leq j,k \leq m\}},
    \end{equation}
    and an angle $\tilde{\varphi}$, such that
    \begin{equation}\label{eq:full_trajectory}
      X(T)\Pi_m\Psi = \mathrm{e}^{i\tilde{\varphi}}\Pi_m\Phi_j
    \end{equation}
    where $\Pi_m$ is the projection
    \begin{equation*}
      \Pi_m:\hilb\rightarrow\mathbb{C}^m,\quad \Psi\in\hilb\mapsto\left(\Braket{\Phi_1,\Psi},\ldots,\Braket{\Phi_m,\Psi}\right)\in\mathbb{C}^m.
    \end{equation*}

     Let $\mu>0$. Its value will be chosen later in the proof.
    Then, by~\cite[Proposition 5.6]{BoscainCaponigroChambrionEtAl2012}, for every $r>0$ there exists certain $\nu>\frac{2}{5}$, cf.\ \cite[Remark 4.4]{BoscainCaponigroChambrionEtAl2012} and a piecewise constant function $\tilde{u}:[0,T]\rightarrow[1/r,+\infty)$ such that, for all $k\in\mathbb{N}$ and $t\in[0,T]$,
    \begin{equation*}
      \bigl| |\braket{\Pi_m\Phi_k,X(t)\Pi_m\Psi}| - |\braket{\Phi_k, U_{\tilde{u}}(t,0) \Psi}|\bigr|< \mu,
    \end{equation*}
    where $U_{\tilde{u}}(t,s)$ is the piecewise weak solution of the time reparametrised system of \cite[Proposition~4.2]{BoscainCaponigroChambrionEtAl2012}.
    By Eq.~\eqref{eq:full_trajectory}, the latter equation for $t=T$ and $k = j$ reduces to
    \begin{equation}\label{eq:ineq1}
      \big| 1 - |\langle \Phi_j, U_{\tilde{u}}(T,0)\Psi \rangle| \big| < \mu,
    \end{equation}
    which implies
    \begin{equation*}
      1 - \mu < |\langle \Phi_j, U_{\tilde{u}}(T,0) \Psi \rangle|.
    \end{equation*}
    Now, by the unitarity of $U_{\tilde{u}}(T,0)$,
    \begin{equation*}
      1 = \|U_{\tilde{u}}(T,0)\Psi\|^2 = \left|\Braket{\Phi_j,U_{\tilde{u}}(T,0)\Psi}\right|^2
      + \sum_{k \neq j}\left| \Braket{\Phi_k,U_{\tilde{u}}(T,0)\Psi}\right|^2
    \end{equation*}
    and therefore
    \begin{equation}\label{eq:ineq2}
      \sum_{k \neq j} \left| \Braket{\Phi_k,U_{\tilde{u}}(T,0)\Psi}\right|^2
       = 1 - \left|\Braket{\Phi_j,U_{\tilde{u}}(T,0)\Psi}\right|^2  \leq 2\mu.
    \end{equation}

    Now define the angle $\varphi$ by $\Braket{\Phi_j,U_{\tilde{u}}(T,0)\Psi}=\mathrm{e}^{i\varphi}| \langle \Phi_j, U_{\tilde{u}}(T,0)\Psi \rangle|$, and therefore, using the inequalities~\eqref{eq:ineq1} and~\eqref{eq:ineq2},
    \begin{align*}
      \left\| e^{i\varphi} \Phi_j - U_{\tilde{u}}(T, 0) \Psi \right\|^2
      &= \left| \langle \Phi_j, e^{i\varphi} \Phi_j - U_{\tilde{u}}(T, 0)\Psi \rangle\right|^2
      + \sum_{k \neq j} \left| \langle \Phi_k, e^{i\varphi} \Phi_j - U_{\tilde{u}}(T, 0)\Psi \rangle \right|^2 \\
      &= \left| \mathrm{e}^{i\varphi} - \langle \Phi_j, U_{\tilde{u}}(T, 0) \Psi\rangle\right|^2
      + \sum_{k \neq j} \left|\Braket{\Phi_k,U_{\tilde{u}}(T, 0)\Psi}\right|^2\\
      &< \left|\mathrm{e}^{i\varphi}-\mathrm{e}^{i\varphi}|\langle \Phi_j, U_{\tilde{u}}(T, 0) \Psi\rangle |\right|^2 + 2\mu \\
      &= \big|1 - |\langle \Phi_j,U_{\tilde{u}}(T, 0)\Psi \rangle|\big|^2+ 2\mu \\
      &<\mu^2+2\mu.
    \end{align*}

    Taking $\mu$ such that $\mu^2 + 2\mu < \varepsilon$ shows that the time-reparametrised system is approximately controllable with a piecewise constant control $\tilde{u}:[0,T]\to[\frac{1}{r}, \infty)$ with $T$ bounded by Eq.~\eqref{eq:Time_bound} and $\nu>\frac{2}{5}$. By \cite[Proposition~4.2]{BoscainCaponigroChambrionEtAl2012} this implies the existence of a piecewise constant control $u:[0,\tau]\to[0,r)$ with 
    $$
    \begin{aligned}
      \|u\|_{L_1(0,\tau)} = T &< \frac{5(m - 1)\pi}{4 \min\{|\langle \Phi_j, H_1 \Phi_k \rangle| \colon (j,k) \in S,\, 1 \leq j,k \leq m\}}
    \end{aligned}
    $$
    and associated propagator as in the statement.
  \end{proof}

\subsection{Proof of Theorem~\ref{thm:bounded_control_span}}\label{subsec:proof_chambrion}

  We can finally address the proof of Theorem~\ref{thm:bounded_control_span}. This will be achieved in three steps:
  \begin{enumerate}
    \item first, invoking Proposition~\ref{prop:bounded_control_span}, we will drive the initial state $\Psi_0$ into a fixed eigenstate $e^{i\varphi}\Phi_k$ via some control function $u_1(t)$;
    \item secondly, we will adjust the angle via a free evolution;
    \item finally, again invoking Proposition~\ref{prop:bounded_control_span}, we will drive the resulting state close to the desired target state $\Psi_1$.
  \end{enumerate}

  \begin{proof}[Proof of Theorem~\ref{thm:bounded_control_span}]
    By Proposition~\ref{prop:bounded_control_span} there exists $T_1>0$ and a control function $u_1:[0, T_1] \to [0, r)$, with $\|u_1\|_{L^1} \leq \frac{5(m-1)\pi}{4\min\{|\langle \Phi_j, H_1 \Phi_k \rangle| \colon (j,k) \in S,\, 1 \leq j,k \leq m\}}$, and $\varphi$ such that the associated solution of the Schrödinger equation, $U_{u_1}$, satisfies
    \begin{equation} \label{eq:evol_part1}
      \left\|U_{u_1}(T_1, 0)\Psi_0 - \mathrm{e}^{i\varphi} \Phi_k \right\| < \frac{\varepsilon}{2}.
    \end{equation}	
    Besides, since $H(t) = H_0 + u(t) H_1$ satisfies the hypotheses of Proposition~\ref{prop:bounded_control_span}, then the same applies to $\tilde{H}(t) = -H_0 - u(t) H_1$. Therefore, there also exist $T_2>0$ and $u_2:[0, T_2] \to [0, r)$, with $\|u_2\|_{L^1} \leq \frac{5(m-1)\pi}{4\min\{|\langle \Phi_j, H_1 \Phi_k \rangle| \colon (j,k) \in S,\, 1 \leq j,k \leq m\}}$, and $\tilde{\varphi}$ such that 
    \begin{equation*}
      \left\|\tilde{U}_{u_2}(T_2, 0)\Psi_1 - \mathrm{e}^{i\tilde{\varphi}} \Phi_k \right\| < \frac{\varepsilon}{2},
    \end{equation*}		
    where $\tilde{U}_{u_2}(t,s)$ denotes the unitary propagator solving the Schrödinger equation with Hamiltonian $\tilde{H}(t)$.
    On the other hand, the Schrödinger equation is invariant under a simultaneous inversion of time and a change of sign of the Hamiltonian, implying that $\tilde{U}_{u_2}(t,s)=U_{u_2}(s,t)=U_{u_2}(t,s)^{-1}$.
    Therefore, also taking into account that the initial time of the evolution is immaterial, the function $u_2(t-s)$ ensures, for any $s>0$,
    \begin{equation} \label{eq:evol_part3}
      \left\|U_{u_2}(T_2+s,s)\left(\mathrm{e}^{i\tilde{\varphi}} \Phi_k\right) - \Psi_1\right\| < \frac{\varepsilon}{2}.
    \end{equation}

    Finally, there exists $\tau>0$ such that the free evolution generated by $H_0$ satisfies
    \begin{equation} \label{eq:evol_part2}
      U_0(T_1 + \tau, T_1) \mathrm{e}^{i\varphi} \Phi_k
      = \mathrm{e}^{i\lambda_j \tau} \mathrm{e}^{i\varphi} \Phi_k
      = \mathrm{e}^{i \tilde{\varphi}} \Phi_k.
    \end{equation}
    Now let $T = T_1 + T_2 + \tau$, and define $u: [0, T] \to [0, r)$ such that
    \begin{equation*}
      u(t) = \begin{cases}
        u_1(t),  & 0 \leq t < T_1 \\
        0,       & T_1 \leq t < T_1 + \tau \\
        u_2(t-\tau-T_1),  & T_1 + \tau \leq t \leq T.
      \end{cases}
    \end{equation*}
    By construction, $\|u\|_{L^1} \leq \frac{5(m-1)\pi}{2\min\{|\langle \Phi_j, H_1 \Phi_k \rangle| \colon (j,k) \in S,\, 1 \leq j,k \leq m\}}$. The propagator $U_u(t,s)$ associated with the control $u$ satisfies $U_u(T,0)=U_{u_2}(T,T_1+\tau)U_0(T_1+\tau,T_1)U_{u_1}(T_1,0)$, so that, by Eqs.~\eqref{eq:evol_part1}, \eqref{eq:evol_part3}, \eqref{eq:evol_part2} and the unitarity of the propagators, the evolution with control $u(t)$ finally satisfies
    $$
    \begin{aligned}
      \|U_u(T, 0) \Psi_0 - \Psi_1\| & \leq \norm{\Psi_1 - U_{u_2}(T,T_1+\tau)\mathrm{e}^{i\tilde{\varphi}}\Phi_k} + \dots \\
        &\quad\quad \dots + \norm{\mathrm{e}^{i\tilde{\varphi}}\Phi_k - U_0(T_1+\tau, T_1)\mathrm{e}^{i\varphi}\Phi_k} + 
        \norm{\mathrm{e}^{i\varphi}\Phi_k - U_1(T_1,0)\Psi_0}  <      \varepsilon. \qedhere\\ 
    \end{aligned}
    $$%
  \end{proof}

\section{Proof of Theorems~\ref{thm:main1}--\ref{thm:compact}}\label{sec:proof}

We can now proceed with the proof of the controllability results claimed in Section~\ref{sec:results}. This will be achieved by combining Theorem~\ref{thm:approximating_families_controllability} with Theorem~\ref{thm:bounded_control_span}. The first step in this direction will be to show that, given a form bilinear control system $(H_0,h_1,r)$ admitting a connectedness chain $S$ and an approximating family $\{h_1^{(n)}\}_{n\in\mathbb{N}}$, cf.\ Definition \ref{def:approx_family}, then the latter can be chosen in such a way that $S$ is also a connectedness chain for the form bilinear control systems $(H_0,h_1^{(n)},r)$. That is the content of the following result.

\begin{lemma} \label{lemma:connectedness_approximation}
Let $(H_0, h_1, r)$ be a form bilinear quantum control system, and let S be a connectedness chain for $(H_0, h_1, r)$. 
For any approximating
 family $\{h_1^{(n)}\}_{n\in\mathbb{N}}$ there is an approximating family $\{\tilde{h}_1^{(n)}\}_{n\in\mathbb{N}}$ such that, for every $n$, $(H_0, \tilde{h}_1^{(n)}, r)$ admits $S$ as a connectedness chain. 
Moreover, if the original approximating family $\{h_1^{(n)}\}$ is regular, then $\{\tilde{h}_1^{(n)}\}_{n\in\mathbb{N}}$ is also regular.
\end{lemma}
\begin{proof}
  By Proposition~\ref{prop:approximations}, we can choose the approximating family so that $(H_0,h_1^{(n)},r)$ is a form bilinear control system for all $n\in\mathbb{N}$. Define
  \begin{equation*}
    S_0^{(n)}:=\left\{(s,s')\in S:\;h_1^{(n)}(\Phi_s,\Phi_{s'})=0\right\},
  \end{equation*}
  that is, $S_0^{(n)}$ is the set of all pairs in the connectedness chain $S$ that are not part of chains with respect to $h_1^{(n)}$. If $S_0^{(n)}=\emptyset$, define $\tilde{h}_1^{(n)}=h_1^{(n)}$. If not, since $S_0^{(n)}$ is either finite or countably infinite, we can pick a bijective function
  \begin{equation*}
    b_n:S_0^{(n)}\rightarrow\{0,1,2,\dots,|S_0^{(n)}|-1\},
  \end{equation*}
  where $|S_0^{(n)}|$ is the cardinality of the set $S_0^{(n)}$ (it is understood that, if $S_0^{(n)}$ is countably infinite, then $b_n$ is a bijection between $S_0^{(n)}$ and $\mathbb{N}$). Define the operator $P^{(n)}$ on $\hilb$ by
  \begin{equation*}
    P_n:=\frac{1}{8n}\sum_{(s,s')\in S_0^{(n)}}2^{-b_n(s,s')}\bigl[
      \braket{\Phi_{s'}, \cdot}{\Phi_{s}}+\braket{\Phi_{s}, \cdot}{\Phi_{s'}}\bigr
    ].
  \end{equation*}
  Since $\mathcal{B}(\hilb)\subset\mathcal{B}(\hilb^+,\hilb^-)$, then $P_n$ is also an element of $\mathcal{B}(\hilb^+,\hilb^-)$ with
$$
\begin{aligned}
      \left\|P_n\right\|_{+,-}&:=\sup_{\Phi,\Psi\in\hilb^+}\frac{|\Braket{\Psi,P_n\Phi}|}{\|\Psi\|_+\|\Phi\|_+}\leq\frac{1}{4n}\sum_{(s,s')\in S_0^{(n)}}2^{-b_n(s,s')}\sup_{\Phi,\Psi\in\hilb^+}\frac{\|\Psi\|\|\Phi\|}{\|\Psi\|_+\|\Phi\|_+} \leq\frac{1}{2n},\\
\end{aligned}
$$
  where in the last step we used the bijectivity of $b_n$. This shows that $\|P_n\|_{+,-}\to0$ as $n\to\infty$. As such, defining the form 
  \begin{equation}\label{eq:S_connected_perturbation}
    \tilde{h}_1^{(n)}(\Psi,\Phi):=h_1^{(n)}(\Psi,\Phi)+\Braket{\Psi,P_n\Phi},
  \end{equation}
  $\{\tilde{h}_1^{(n)}\}_{n\in\mathbb{N}}$ is still an approximating family for the form bilinear control system $(H_0,h_1,r)$. 
  Furthermore, $S$ is a connectedness chain for the form bilinear control system $(H_0,\tilde{h}_1^{(n)},r)$. Indeed, since $S$ is a connectedness chain for $(H_0,h_1,r)$, for every $j,\ell\in\mathbb{N}$ there exists 
  \begin{equation*}
    (s_0,s_1),(s_1,s_2),\ldots,(s_{k-1},s_k)\in S, 
  \end{equation*}
  with $s_0=j$ and $s_k=\ell$, such that $h_1(\Phi_{s_a}, \Phi_{s_{a+1}}) \neq 0$ for all $a=0,\dots,k-1$. By construction, given $n\geq n_0$, and noting that
  \begin{equation*}
    \Braket{\Phi_{s_a},P_n\Phi_{s_{a+1}}}=\begin{cases}
      0,&(s_a,s_{a+1})\notin S_0^{(n)}\\
      \frac{1}{8n}\left[2^{-b_n(s_a,s_{a+1})}+2^{-b_n(s_{a+1},s_{a})}\right],&(s_a,s_{a+1})\in S_0^{(n)}
    \end{cases}
  \end{equation*}
  we finally have
  \begin{equation*}
    \tilde{h}_1^{(n)}(\Phi_{s_a},\Phi_{s_{a+1}})=\begin{cases}
      h_1^{(n)}(\Phi_{s_a},\Phi_{s_{a+1}}),&\text{if }(s_a,s_{a+1})\notin S_0^{(n)}\\
      \Braket{\Phi_{s_a},P_n\Phi_{s_{a+1}}},&\text{if }(s_a,s_{a+1})\in S_0^{(n)}
    \end{cases}\neq0,
  \end{equation*}
  thus proving the first claim by dropping the first $n_0$ elements if necessary. By Eq.~\eqref{eq:S_connected_perturbation} it follows that $\{\tilde{h}_1^{(n)}\}_{n\in\mathbb{N}}$ is  regular if $\{h_1^{(n)}\}_{n\in\mathbb{N}}$ is.

\end{proof}

We can finally prove the main results of the article.

\begin{proof}[Proof of Theorem~\ref{thm:main1}]
  Let $(H_0,h_1,r)$ satisfy Assumptions~\ref{assump:form} and admit a regular approximating family; that is, there exists a family of linear operators $\{H_1^{(n)}\}_{n\in\mathbb{N}}$ on $\hilb$ such that, from one hand, $(H_0,H_1^{(n)},r)$ is a bilinear control system, cf.~Definition~\ref{def:bilinear_control_system}; and, from the other hand, there exists a bounded form $h_1^{(n)}:\hilb^+\times\hilb^+\rightarrow\mathbb{C}$ satisfying $h_1^{(n)}(\Psi,\Phi):=\braket{\Psi,H_1^{(n)}\Phi}$ for all $\Psi\in\hilb^+$ and $\Phi\in\dom H_1^{(n)}$ that define an approximating family for $(H_0,h_1,r)$.
  The latter claim implies, by Proposition~\ref{prop:approximations}, that $(H_0,h_1^{(n)},r)$ is a form bilinear control system for $n$.
  Furthermore, by Lemma~\ref{lemma:connectedness_approximation}, there is no loss of generality in assuming that $S$ is also a connectedness chain for each of these systems.
  Thus, Theorem~\ref{thm:bounded_control_span} applies to $(H_0,H_1^{(n)},r)$: for every fixed integer $m>0$, every $\Psi_0, \Psi_1 \in \lspan{\Phi_j \mid 1 \leq j \leq m}$ with $\|\Psi_0\| = \|\Psi_1\|$ and every $\varepsilon > 0$, there exists a time $T_n$ and a piecewise linear control $u_n: [0, T_n] \to [0, r)$ such that the propagator $U_{u_n}(t,s)$, which is a piecewise weak solution of the associated Schrödinger equation, satisfies
  $\|U_{u_n}(T_n, 0)\Psi_0 - \Psi_1\| < \varepsilon$ and
  \begin{equation*}
    \|u_n\|_{L_1} \leq \frac{5(m - 1)\pi}{2 \min\{|h_1^{(n)}\left(\Phi_j, \Phi_k\right) \colon (j,k) \in S,\, 1 \leq j,k \leq m\}}.
  \end{equation*}
  But, since $h_1^{(n)}\left(\Phi_j,\Phi_k\right)\to h_1\left(\Phi_j,\Phi_k\right)$, the $L^1$ norm of each control function $u_n$ is bounded from above uniformly in $n$; besides, since $u_n$ is piecewise constant, its piecewise derivatives vanish.
  Thus, Theorem~\ref{thm:approximating_families_controllability} applies, concluding the proof.
\end{proof}

\begin{proof}[Proof of Theorem~\ref{thm:main2}]
  Let $\varepsilon>0$ and $\Psi_0,\Psi_1\in\hilb$ with $\|\Psi_0\|=\|\Psi_1\|$.
  Since $\{\Phi_j\}_{j \in \mathbb{N}}$ is a complete orthonormal set of $\hilb$, there exists $m > 0$, $\Psi_0^{(m)}, \Psi_1^{(m)} \in \lspan{\Phi_j \mid 1 \leq j \leq m}$ such that
  \begin{equation*}
    \|\Psi_0 - \Psi_0^{(m)}\| < \varepsilon / 3 \quad \text{and} \quad \|\Psi_1 - \Psi^{(m)}_1\| < \varepsilon / 3.
  \end{equation*}
  The statement now follows from Theorem~\ref{thm:main1} applied to the states $\Psi_0^{(m)}, \Psi^{(m)}_1$ and a standard $\varepsilon/3$ argument.
\end{proof}

\begin{proof}[Proof of Theorem~\ref{thm:compact}]
  Let $h_1$ define a compact operator $H_1\in\mathcal{B}(\hilb^+,\hilb^-)$. Then Prop.~\ref{prop:compact} implies that $(H_0,h_1,r)$ admits a regular approximating family, whence the claim follows from Theorem~\ref{thm:main2}.
\end{proof}

%%%%%%%%%%%%%%%%%%%%%%%%%%%%%%%%%%%%%%%%%%%%%%%%%%%%%%%%%%
%%%%%%%%%%%%%%%%%%%%%%%%%%%%%%%%%%%%%%%%%%%%%%%%%%%%%%%%%%
%%%%%%%%%%%%%%%%%%%%%%%%%%%%%%%%%%%%%%%%%%%%%%%%%%%%%%%%%%
%%%%%%%%%%%%%%%%%%%%%%%%%%%%%%%%%%%%%%%%%%%%%%%%%%%%%%%%%%

\section{Example:  Quantum particle in a one-dimensional box controlled by a point-like interaction}\label{sec:ExampleDiracDelta}

In this example we shall consider the situation of a nonrelativistic quantum particle trapped in a one-dimensional box that has a delta-like interaction with tuneable strength $\mu\in\mathbb{R}$ in the centre of the box. Mathematically, this is described by the Laplace operator on $\Omega = [0,\frac{1}{2}]\cup[\frac{1}{2},1]$ 
with boundary $\partial\Omega = \{0, \frac{1}{2}^-,\frac{1}{2}^+,1 \}$ subject to self-adjoint boundary conditions at $\{\frac{1}{2}^-, \frac{1}{2}^+\}$ that implement the point-like interaction with variable strength, and Dirichlet boundary conditions at $\{0,1\}$. The Hilbert space of the system is $\mathcal{H} = L^2[0,\frac{1}{2}]\oplus L^2[\frac{1}{2},1]$ and we will use the notation $\mathcal{H}\ni \Psi = \Psi^-\oplus\Psi^+$ with $\Psi^-\in L^2[0,\frac{1}{2}]$ and $\Psi^+\in L^2[\frac{1}{2},1]$. This corresponds to the following dynamical system:
$$\frac{\mathrm{d}}{\mathrm{d}t}\Psi = -i\frac{\mathrm{d}^2}{\mathrm{d}x^2}\Psi$$
with parameter-dependent  domain 
\begin{equation}\label{eq:domainexample}
  \mathcal{D}(\mu) = \left\{ \Psi \in \mathcal{H}^2\left([0,{\textstyle\frac{1}{2}}]\right)\oplus \mathcal{H}^2\left([{\textstyle\frac{1}{2}},1]\right) \left| 
  \begin{array}{c}
    \Psi(0)=\Psi(1)=0,\\ \\
    \Psi^-(\frac{1}{2}) =  \Psi^+(\frac{1}{2}), \\ \\
    \displaystyle
    \frac{\mathrm{d}\Psi^-}{\mathrm{d}x}({\textstyle\frac{1}{2}}) - \frac{\mathrm{d}\Psi^+}{\mathrm{d}x}({\textstyle\frac{1}{2}}) = \mu \Psi^-({\textstyle\frac{1}{2}})
  \end{array}
  \right\}\right.,
\end{equation}
where $\mathcal{H}^2$ stands for the Sobolev space of order 2. The strength of the interaction is determined by the parameter $\mu\in\mathbb{R}$ which will be the control $\mu\colon[0,T]\to\mathbb{R}$.

The time-dependent quadratic form associated with the self-adjoint operator above, cf.\ \cite{ibort2015self, BalmasedaLonigroPerezPardo2023b}, has constant form domain $\mathcal{H}^+ = \left\{ \Psi \in \mathcal{H}^1_0([0,1]) \right\}$ and is given by
$$h_t(\Phi,\Psi) = \int_0^1 \overline{\frac{\mathrm{d}\Phi}{\mathrm{d}x}}\frac{\mathrm{d}\Psi}{\mathrm{d}x} \mathrm{d}x + \mu(t) \overline{\Phi}({\textstyle\frac{1}{2}})\Psi({\textstyle\frac{1}{2}})=:h_0(\Phi,\Psi) + \mu(t)h_1(\Phi,\Psi),\quad \Phi,\Psi\in\hilb^+.$$
We have defined $h_0$ and $h_1$ in the obvious way and $\mathcal{H}^1_0$ stands for the Sobolev space of order one whose elements vanish at the boundary. 

For any $c\in(0,1)$ and $f\in\hilb^1_0[0,1]$ it holds 
$$
|f(c)|^2 = \int_0^c \frac{\mathrm{d}}{\mathrm{d}x}|f(x)|^2\mathrm{d}x 
  \leq 2\left| \int_0^c f'(x)f(x) \mathrm{d}x\right| 
  \leq \varepsilon\norm{f'}_{L^2[0,1]}^2 + \frac{1}{\varepsilon}\norm{f}_{L^2[0,1]}^2,
$$
where we have used Young's inequality with epsilon. This shows that $h_1$ is infinitesimally relatively form bounded with respect to $h_0$. Hence, the above dynamical system determines a form bilinear control system $(H_0, h_1,r)$, see Definition~\ref{def:form_linear_control_system}, with form domain $\mathcal{H}^+$ and $H_0$ the Dirichlet Laplacian on $[0,1]$ for any $r>0$.

Notice also that $h_1$ is determined by a finite rank operator $H_1:\mathcal{H}^+\to\mathcal{H}^-$ and therefore we can apply Theorem~\ref{thm:compact}. For that we need to show that $(H_0, h_1,r)$ admits a non-resonant connectedness chain, cf.\ Assumption~\ref{assump:form}.

Let $\{\Phi_k\}$ be the orthonormal basis of eigenfunctions of the Dirichlet Laplacian $H_0$ and $\{E_k\}$ the corresponding eigenvalues, given by 
$$E_k = k^2\pi^2,\quad \Phi_k(x) = \sqrt{2}\sin\left[ k\pi\left( x +\frac{1}{2}\right)\right].$$
Straightforward computations show that $h_1(\Phi_k,\Phi_l)=\overline{\Phi_k}(\frac{1}{2})\Phi_l(\frac{1}{2})=0$ whenever $k$ or $l$ are even. This means that with a point-like interaction located in the middle of the interval one cannot achieve global approximate controllability in $\hilb$ as there is no connectedness chain. However, restricting to the subspace of eigenfunctions with even parity, that is $\hilb_{\text{even}}=\operatorname{span}\{\Phi_k\}_{k\text{ odd}}$, one has that $h_1(\Phi_k,\Phi_l)\neq0$ and hence, by Theorem~\ref{thm:compact}, the system will be approximately controllable if the non-resonant condition on the eigenvalues is met. We will discuss how to achieve controllability in larger subspaces at the end of this section. 

The eigenvalues of the Dirichlet Laplacian are not resonance free. Nevertheless one can avoid this difficulty by proceeding as done in \cite{ChambrionMasonSigalottiEtAl2009, balmaseda2022global} as follows. Let $\eta>0$ and consider $h_0(\eta) = h_0 + \eta h_1$. Since $h_1$ is infinitesimally relatively bounded with respect to $h_0$, by Theorem~\ref{thm:klmn} the form $h_0(\eta)$ is closed, bounded from below and has the same form domain as $h_0$. Moreover, $h_1$ is also infinitesimally relatively bounded  with respect to $h_0(\eta)$.
Let $H_0(\eta)$ be the unique self-adjoint operator associated with $h_0(\eta)$. Hence $(H_0(\eta),h_1,r)$ defines a form bilinear control system and it is approximately controllable with controls $\tilde{u}\in(a,b)$ if and only if  the original system is approximately controllable with controls ${u}\in(a+\eta,b+\eta)$.

Moreover, $h_0(\eta)$ is a holomorphic family of type B as defined in \cite{Kato1995}. This implies that the self-adjoint operator $H_0(\eta)$ associated with $h_0(\eta)$ has compact resolvent either for every $\eta\in\mathbb{R}$ or for none. Since $H_0(0)$ is the Dirichlet Laplacian on a compact manifold we are in the former case. Moreover, the eigenfunctions and eigenvalues are analytic functions of the parameter $\eta$.

\begin{lemma}\label{lem:perturbed_nonresonance}
Let $S\subset\mathbb{N}^2$ be a connectedness chain for the form bilinear control system $(H_0(0),$ $h_1, r)$. Then for all $\eta_0>0$ there exists $0<\eta<\eta_0$ such that $S$ is a non-resonant connectedness chain for the form bilinear control system $(H_0(\eta), h_1, r)$.
\end{lemma}

The proof is completely analogous to that of \cite[Lemma 4.4]{balmaseda2022global} and uses the analiticity of the eigenfunctions and eigenvalues with respect to $\eta$. The only differences with respect to this case are the particular expressions for the first and second order perturbations of the eigenfunctions and eigenvalues. We will only provide the details which differ from the referenced proof. 

\begin{proof}

The eigenvalues associated with eigenfunctions with even parity of the unperturbed system are given by $E_{2j+1} = \pi^2(2j+1)^2$, $j\in\mathbb{N}_0$, and from standard perturbation theory
$$
  \left.\frac{\mathrm{d}}{\mathrm{d}\eta}E_{2l+1}\right|_{\eta=0} = h_1(\Phi_{2l+1},\Phi_{2l+1}) = 2.
$$
The resonance condition involves differences of the eigenvalues. Since the first order term is uniform for all the eigenvalues we need to look for the second order correction to guarantee that the perturbation leads to a lifting of the resonances. The second order correction is given by:
$$
  \left.\frac{\mathrm{d}^2}{\mathrm{d}\eta^2}E_{2l+1}\right|_{\eta=0} = \sum_{j\neq l} \frac{|h_1(\Phi_{2l+1},\Phi_{2l+1})|^2}{E_{2j+1}-E_{2l+1}}.
$$

After a straightforward calculation we get
$$
  \left.\frac{\pi^2}{4} \frac{\mathrm{d}^2}{\mathrm{d}\eta^2}E_{2l+1}\right|_{\eta=0} = \sum_{j\neq l} \frac{1}{(2j+1)^2 - (2l+1)^2} = \frac{1}{4(2l+1)}  \sum_{j\neq l} \left(\frac{1}{j-l} - \frac{1}{j+l+1}\right).
$$
The sum at the right hand side can be calculated explicitly to get
$$
  \begin{aligned}
    \sum_{j\neq l} \left(\frac{1}{j-l} - \frac{1}{j+l+1} \right)
      &= \sum_{j=0}^{l-1} \left(\frac{1}{j-l} - \frac{1}{j+l+1}\right ) + \sum_{j=l+1}^\infty \left(\frac{1}{j-l} - \frac{1}{j+l+1}\right) \\
      &= \frac{1}{2l+1} - \sum_{j=0}^\infty     \frac{1}{j+l+1}  + \sum_{j=0}^{l-1}\frac{1}{j-l} + \sum_{j=l+1}^{\infty}\frac{1}{j-l} \\
      &= \frac{1}{2l+1}.
  \end{aligned}
$$
This results in 
$$
  \left.\frac{\mathrm{d}^2}{\mathrm{d}\eta^2}E_{2l+1}\right|_{\eta=0} = \frac{1}{\pi^2} \frac{1}{(2l+1)^2}.
$$
Hence, apart from the uniform factor, $\frac{1}{\pi^2}$, we get the same non-resonant condition as in \cite[Eq.~(4.27)]{balmaseda2022global}.
\end{proof}

In summary, there exists $\eta>0$, that can be chosen arbitrarily small, such that for any $r>0$
$(H_0(\eta), h_1,r)$ is approximately controllable in the subspace associated with the eigenfunctions with even parity of the Dirichlet Laplacian, $\hilb_{\text{even}}$. Therefore, for any $r>0$, the form bilinear quantum control system $(H_0, h_1,r)$ is approximately controllable in the same space, as we wanted to show. Hence we have proven the following theorem:

\begin{theorem}\label{thm:Diracdeltacontrollable}
Consider the quantum system defined by the time-dependent Hamiltonian given by the Laplacian on the intervals $\Omega=[0,\frac{1}{2}]\cup[\frac{1}{2},1]$ with time-dependent domain $\mathcal{D}(\mu(t))$ given in Eq.~\eqref{eq:domainexample}. For any $\varepsilon>0$, $\Psi_0, \Psi_1\in\hilb_{\text{even}}$ with $\norm{\Psi_0}=\norm{\Psi_1}$ and $r>0$ there exists a piecewise constant function $\mu:[0,T]\to[0,r)$ such that the Schrödinger equation is well-posed and admits a piecewise weak solution $U_\mu(t,s)$ such that 
$$
  \norm{\Psi_1 - U_\mu(T,0)\Psi_0}<\varepsilon.
$$
\end{theorem}

We have proven controllability in a closed and non-dense subspace of the drift operator $H_0$. The system described in this situation is invariant under parity transformations \cite{IbortLledoPerezPardo2015a, ibort2015theory} and the appearance of invariant subspaces is expected, hence preventing global approximate controllability. Although a thorough study must be carried out, the results above can be extended using the framework described in this article. By displacing the delta perturbation it should be possible to enlarge the closed subspace where approximate controllability is achieved. The connectedness of the chain is broken for those eigenfunctions that have nodes at the position of the delta interaction. Since the eigenfunctions are trigonometric functions, the set of eigenfunctions having nodes at a given position can be chosen arbitrarily small by choosing the position of the interaction carefully. Moreover, if the point-like interaction is at an irrational position, then $h_1(\Phi_k,\Phi_j)\neq 0$ for all $j,k$. The non-resonance conditions should be studied in these cases. 

A limiting situation is having the point-like interaction located at the endpoint of the interval. This situation can also be studied with the notions introduced in this article, and it corresponds to a particle moving freely in the interval with Dirichlet boundary conditions at one side and time-dependent Robin boundary conditions at the other one.

\vspace{1em}
{\small \textit{Acknowledgments.}{
  A.B. and J.M.P.P. acknowledge support provided by the ``Agencia Estatal de Investigación (AEI)'' Research Project PID2020-117477GB-I00, by the QUITEMAD Project P2018/TCS-4342 funded by the Madrid Government (Comunidad de Madrid-Spain) and by the Madrid Government (Comunidad de Madrid-Spain) under the Multiannual Agreement with UC3M in the line of ``Research Funds for Beatriz Galindo Fellowships'' (C\&QIG-BG-CM-UC3M), and in the context of the V PRICIT (Regional Programme of Research and Technological Innovation).
  J.M.P.P acknowledges financial support from the Spanish Ministry of Science and Innovation, through the ``Severo Ochoa Programme for Centres of Excellence in R\&D'' (CEX2019-000904-S). 
  A.B.\ acknowledges financial support from the Spanish Ministry of Universities through the UC3M Margarita Salas 2021-2023 program (``Convocatoria de la Universidad Carlos III de Madrid de Ayudas para la recualificación del sistema universitario español para 2021-2023''), and from ``Universidad Carlos III de Madrid'' through Ph.D.\ program grant PIPF UC3M 01-1819, UC3M mobility grant in 2020 and from the EXPRO grant No.\ 20-17749X of the Czech Science Foundation.
  D.L. acknowledges financial support from European Union--NextGenerationEU (CN00000013 -- ``National Centre for HPC, Big Data and Quantum Computing'') and was partially supported by the Italian National Group of Mathematical Physics (GNFM-INdAM) and by Istituto Nazionale di Fisica Nucleare (INFN) through the project QUANTUM. He also thanks the Department of Mathematics at ``Universidad Carlos III de Madrid'' for its hospitality.}
  \printbibliography
}
\end{document}